%% file: DEKI.tex
\documentclass[final,hidelinks,onefignum,onetabnum]{siamart220329} 
\usepackage{Basic_siam,setspace}
\usepackage{color}

\input{siam_shared}

\newcommand{\bx}{{\bm{x}}}

\definecolor{darkred}{rgb}{.6,0,0}
\definecolor{darkblue}{rgb}{0,0,.7}
\definecolor{darkgreen}{RGB}{0,100,0}
\definecolor{darkbrown}{rgb}{0.8,0.4,0.4}

\title{Dropout ensemble Kalman inversion for high dimensional inverse problems\thanks{Submitted to the editors DATE.}}

\ifpdf
\hypersetup{
  pdftitle={Dropout ensemble Kalman inversion for high dimensional inverse problems},
  pdfauthor={Shuigen Liu, Sebastian Reich and Xin T. Tong}
}
\fi

\begin{document}

\maketitle

\begin{abstract}
    Ensemble Kalman inversion (EKI) is an ensemble-based method to solve inverse problems. 
    Its gradient-free formulation makes it an attractive tool for problems with involved formulation. 
    However, EKI suffers from the “subspace property”, i.e., the EKI solutions are confined in the subspace spanned by the initial ensemble. It implies that the ensemble size should be larger than the problem dimension to ensure EKI’s convergence to the correct solution. Such scaling of ensemble size is impractical and prevents the use of EKI in high dimensional problems.
    To address this issue, we propose a novel approach using dropout regularization to mitigate the subspace problem. We prove that dropout-EKI converges in the small ensemble settings, and the computational cost of the algorithm scales linearly with dimension. We also show that dropout-EKI reaches the optimal query complexity, up to a constant factor. Numerical examples demonstrate the effectiveness of our approach.
\end{abstract} 



\section{Introduction} 
Inverse problems arise in various scientific and engineering problems \cite{MR2652785,MR4226141,MR3826506,MR1408680}, where one is interested in recovering unknown parameters from indirect and often noisy observations. Mathematically, the relationship between the unknown parameter, denoted by $u\in\mR^{d_u}$,  and the observed data, denoted by $y\in \mR^{d_y}$, can often be described by the following model: 
\begin{equation}
\label{eq:InverseProblem}
y = \mcG (u) + \eta.
\end{equation}
Here $\mcG:\mR^{d_u} \gto \mR^{d_y}$ is the forward map. It describes the physical law behind the observation procedure, and it is often corrupted by some observation noise, which is modeled by $\eta\in\mR^{d_y}$ in \cref{eq:InverseProblem}. 

Directly solving \cref{eq:InverseProblem} is often problematic, since the inverse of $\mcG$ might not be well-defined, and we need to mitigate the corruption from noise. A standard way to handle these issues is regularization \cite{MR3826506,MR1408680}. In this paper, we will focus on the well-known Tikhonov regularization \cite{MR4089506} approach, which is also referred to as $l_2$ or ridge regularization. Mathematically, this involves finding the minimizer $u^*$ of the loss function
\begin{equation}
\label{eq:TikReg}
l(u) = \frac{1}{2} \normo{y - \mcG (u)}^2 + \frac{1}{2} \normo{\mcC_0^{-1/2}u}^2.
\end{equation}
The first part of \cref{eq:TikReg} is the data misfit, which measures the fitness of a candidate solution $u$ when compared with data. The second part of \cref{eq:TikReg} is the regularization term, it measures how well $u$ fits some prior information, e.g. sparsity, intensity and smoothness. Here we choose the Tikhonov regularization, so the prior information is encoded through a positive-definite structural matrix $\mcC_0$. It is worth pointing out that \cref{eq:TikReg} can be seen as negative log-density of the Bayesian posterior of \cref{eq:InverseProblem}, if we assume $u$ has prior $\mathcal{N}(0, \mcC_0)$ and $\eta\sim \mathcal{N}(0, I)$. In this context, the minimizer $u^*$ will become the maximum-a-priori estimator of the posterior density. More details of such connection can be found in \cite{MR2652785,MR3041539,MR3654885}.

Minimizing \cref{eq:TikReg} can be computational challenging in practice. In many applications, $\mcG$ is formulated by a PDE, of which the derivative (or adjoint model) can be too complicate to derive or too large to process.
Ensemble Kalman inversion (EKI) is a derivative-free computational method well-suited for these applications. In short, it employs an ensemble with a Kalman-like formula to avoid the computation of derivatives. By iterating the formula, the EKI ensemble will collapse to a single point which minimizes \cref{eq:TikReg}. EKI has shown impressive inversion skills in various problems \cite{MR3041539,MR3998631}. A series of numerical analyses of EKI have been made recently to justify EKI's performance in linear, nonlinear and stochastic settings  \cite{MR3654885,MR3764752,MR3988266,MR4405495,MR4482475,MR4487558}.

One fundamental problem of EKI is its implementation on high dimensional problems.
This is characterized by the ``subspace property" of EKI \cite{MR3041539,MR3764752}, which indicates that all ensemble members remains in the linear subspace $\mcV_0$ spanned by the initial ensemble. This implies that if we want to obtain $u^*$, $\mcV_0$ needs to be rich enough and the ensemble size needs to surpass $d_u$ in general. But such scaling in ensemble size is impractical for large-scale problems. In practice, this problem can be mitigated by exploiting certain specific inverse problem structures. For example, \cite{chada2018parameterizations} discusses how to select the best $\mcV_0$ so $u^*$ is likely to be in $\mcV_0$. In another recent work,  \cite{MR4580673} considers localization techniques, which can leverage the spatial dependence within an inverse problems. On the other hand, these methods all require certain prior knowledge of the solution structures, so they cannot be universal. 

In this paper, we propose to tackle the subspace issue using dropout technique originated from the deep learning literature \cite{JMLR:v15:srivastava14a,pmlr-v28-wan13,NIPS2015_bc731692,pmlr-v48-gal16}.
In the training process of deep neural networks, the dropout technique randomly dropout neurons to avoid overfitting. In theory, it is usually understood as a regularization method \cite{JMLR:v15:srivastava14a,pmlr-v28-wan13,10.5555/2999611.2999651} or a Bayesian neural network approximation \cite{pmlr-v48-gal16}. In \cite{MR4568201}, dropout is proposed as a heuristic for EKI when implemented on high dimensional logistic regression problems. The idea is justified in \cite{MR4568201} with numerical evidences, but there is not much rigorous understanding.
In this paper, we propose a new formulation of  dropout EKI (DEKI), so it can be applied to high dimensional general nonlinear inverse problems with a fixed ensemble size.  

Moreover, we prove that by properly adapting the step sizes, DEKI scheme converges exponentially fast for linear problems, or strongly convex nonlinear problems in the final stage. The overall computational cost of DEKI scales linearly with dimension $d_u$, which is optimal for high dimensional problems. We also investigate the query complexity of the zeroth order algorithms for inverse problems, and show that the DEKI scheme achieves optimal query complexity up to a constant in the problem class we consider.

The paper is organized as follows. In \cref{sec:DEKIDerive}, we derive the DEKI scheme and compare it with related methods. In \cref{sec:analysis}, we provide the convergence analysis of the DEKI, showing the exponential convergence rate. The query complexity is also analyzed. In \cref{sec:numerics}, we show the numerical performance of DEKI on two examples, confirming our theoretic results.

\subsection{Notations} \label{sec:notations}
For two matrices $A,B \in \mR^{m\times n}$, denote the Hadamard product $C=A\circ B \in \mR^{m \times n}$ where $C_{ij} = A_{ij} B_{ij}$. For two positive semi-definite (PSD) matrices $A,B\in\mR^{m\times m}$, denote $A\succeq B$ iff $A-B$ is PSD. PSD matrices admit standard eigenvalue decomposition $A = Q\Lambda Q\matT $, where $\Lambda$ is a diagonal matrix and $Q$ is an orthogonal matrix. Define the function of PSD matrix $f(A)$ by
\[
f(A) = Q ~\diag\{f(\Lambda_{11}),\dots, f(\Lambda_{mm})\} ~Q\matT .
\]
For sequential random variables $X_1,X_2,\dots$, denote $\mE_n$ as the conditional expectation conditioned on the information up until $n$, i.e. $\mE_n [Y] = \mE[Y|\mcF_n]$ where $\mcF_n $ is the $\sigma$-algebra generated by $X_1,\dots,X_n$.

\section{Dropout EKI scheme}
\label{sec:DEKIDerive}
In this section, we start with a brief review of EKI formulation designed for \cref{eq:InverseProblem}. We then implement the  dropout technique and derive DEKI. Finally we provide some comparison of DEKI with related EKI schemes.

\subsection{Review: EKI and subspace property}
When minimizing the regularized loss function \cref{eq:TikReg}, it is instructive to formulate it as
\begin{equation}
\label{eq:RegProblem}
    l(u) = \frac{1}{2}\norm{z - \mcH(u)}^2, \quad \text{where  } \mcH (u) = 
    \begin{bmatrix}
        \mcG (u) \\
        \mcC_0^{-1/2} u
    \end{bmatrix}, ~ z = 
    \begin{bmatrix}
        y \\
        0
    \end{bmatrix}.
\end{equation}
\cite{MR4089506} investigated applying vanilla EKI from \cite{MR3041539} on \cref{eq:RegProblem}. In particular, if we denote the ensemble as $\{u_n^{(j)}\}_{j=1}^J$, the (Tikhonov) EKI scheme reads
\begin{equation}
\label{eq:EKI0}
    u_{n+1}^{(j)} = u_{n}^{(j)} + h_nC_{n}^{uz} ( h_nC_{n}^{zz} + I )^{-1} ( z - \mcH(u_{n}^{(j)})),
\end{equation}
where $C_{n}^{uz},C_{n}^{zz}$ are the empirical covariance matrices of the deviations, and $I$ denotes the identity matrix. Denote
\begin{equation}
\label{eq:Cov}
    \begin{split}
    & \mean{u}_n = \frac{1}{J} \sum_{j=1}^J u_n^{(j)},\quad \tau_n^{(j)}=u^{(j)}_n-\mean{u}_n, \quad \mean{\mcH(u_n)} = \frac{1}{J} \sum_{j=1}^J \mcH(u_n^{(j)}).\\
     &C_{n}^{uu} = \frac{1}{J-1} \sum_{j=1}^J \tau_n^{(j)} \otimes \tau_n^{(j)},\quad
C_{n}^{uz} = \frac{1}{J-1} \sum_{j=1}^J \tau_n^{(j)} \otimes \Brac{ \mcH(u_n^{(j)}) - \mean{\mcH(u_n)}} \\
    &C_{n}^{zz} = \frac{1}{J-1} \sum_{j=1}^J \Brac{ \mcH(u_n^{(j)}) - \mean{\mcH(u_n)}} \otimes \Brac{ \mcH(u_n^{(j)}) - \mean{\mcH(u_n)}}.
    \end{split}
\end{equation}
The $h_n$ in \cref{eq:EKI0} is the ``stepsize" of the algorithm. The continuous limit formulation can be derived if we fix $h_n\equiv h$ and let $h\to 0$. \cite{MR4405495} has shown that if $h_n$ increases with $n$, the convergence speed can be improved. 

The implementation of EKI involves iterating \cref{eq:EKI0}. In general, the EKI ensemble will collapse onto a single point, and we can use the ensemble mean $\bar{u}_n$ as a candidate minimizer of \cref{eq:InverseProblem}. When we have sufficient ensemble, $J>d_u$, the update \cref{eq:EKI0} can be interpreted as a Gauss-Newton type iteration \cite{MR4405495}.

\subsection{Naive DEKI}
Using mathematical induction, it is not difficult to show that $u_n^{(j)}\in \mcV_0=\text{span}\{u_0^{(j)}\}_{j=1}^J$, which is the ``subspace property" discussed in the introduction.  At a high level, it is not hard to break the ``subspace property", since we can modify \cref{eq:EKI0} so that the new ensembles leave $\mcV_0$ (e.g. adding naive
randomization). But we also need some guidelines to find reasonable modifications to ensure it converges correctly. We derive these guidelines from a variational perspective. In specific, the fixed point of the iterations should be the minimizer  (or a stationary point) of \cref{eq:TikReg}.

For the simplicity of illustration, we will assume the problem is linear, so $\mcH(u)=Hu$ for some matrix $H$. In this case, $C^{uz}_n=C^{uu}_n H\matT $.
Then if $u^{(j)}_n$  is the stationary point of \cref{eq:EKI0}, we find 
\begin{equation}    \label{tmp:opt1}
0=C^{uu}_n H\matT (h_n HC^{uu}_nH\matT +I)^{-1} (z-H u_{n}^{(j)}).     
\end{equation}
Meanwhile, $u^{(j)}_n$ is a minimizer of \cref{tmp:opt1}, iff $0=H\matT (z-H u_{n}^{(j)}).$ This can be guaranteed by \cref{tmp:opt1} if $C^{uu}_n$ has full rank $d_u$. But with vanilla EKI, the rank of $C^{uu}_n$ is at most $J-1$, which is less than $d_u$ for $J\le d_u$.  

Next, we investigate how dropout technique can resolve this issue. Since rank deficiency of $C^{uu}_n$ is the root of the problem, we consider applying dropout on the ensemble deviations $\tau_n^{(j)}=u_n^{(j)}-\mean{u}_n$, which leads to a dropout ensemble $\{\tilde{u}_n^{(j)}\}_{j=1}^{J}$,
\begin{equation}
\label{eq:DropEnsem}
    \tilde{u}_n^{(j)} = \mean{u}_n + \tilde{\tau}_n^{(j)}, \quad \tilde{\tau}_n^{(j)} = \rho_n \circ \tau_n^{(j)}, \quad \rho_n(s) \topsm{\iid}{\sim} \text{Bernoulli}(\lambda).
\end{equation}
Here $1-\lambda \in (0,1)$ is the dropout rate, $\rho_n \in \mR^{d_u}$ is the dropout mask vector, which has components being i.i.d  Bernoulli random variables  indicating the indices of the dropout components, i.e. $ \{ s: \rho_n(s) = 0 \}$. In \eqref{eq:DropEnsem}, $\circ$ denotes the Hadamard product, i.e. $a\circ b (i) = a(i) b(i)$. Note here we take the dropout mask vector $\rho_n$ to be the same over all ensemble members. Such dropout technique follows the idea in \cite{MR4568201}, and we generalize it by introducing the dropout ensemble $\tilde{u}_n^{(j)}$, making it conveniently applicable in general nonlinear problems. In principle, one can also use different dropout mask vectors for different ensemble members. We have tested this variant, but its numerical performance is similar to our version and its analysis harder, so we omit its discussion in this paper. 

Throughout the paper, we will add tilde on symbols related to the dropout ensemble.  In particular, we have the dropout ensemble covariance:
\begin{equation}
\label{eq:dropCov}
    \begin{split}
    &\tilde{C}_{n}^{uu} = \frac{1}{J-1} \sum_{j=1}^J \tilde{\tau}_n^{(j)} \otimes \tilde{\tau}_n^{(j)},\quad\tilde{C}_{n}^{uz} = \frac{1}{J-1} \sum_{j=1}^J \tilde{\tau}_n^{(j)} \otimes \Brac{ \mcH(\tilde{u}_n^{(j)}) - \mean{\mcH(\tilde{u}_n)}}, \\
    &\tilde{C}_{n}^{zz} = \frac{1}{J-1} \sum_{j=1}^J \Brac{ \mcH(\tilde{u}_n^{(j)}) - \mean{\mcH(\tilde{u}_n)}} \otimes \Brac{ \mcH(\tilde{u}_n^{(j)}) - \mean{\mcH(\tilde{u}_n)}}.
    \end{split}
\end{equation}
Replacing the empirical covariance $C_n^{uz}$ and $C_n^{zz}$ in \cref{eq:EKI0} with their dropout versions will lead to the naive DEKI. 
From numerical experiments, we find it can already resolve the ``subspace property", namely it will converge to accurate numerical inversion result even if $J<d_u$, although the convergence speed might be slow. We will add further numerical modifications to accelerate the convergence. 

But before that, it is instructive to provide the intuition why dropout can help with the rank deficiency of sample covariance. Similar to other stochastic algorithms such as stochastic gradient descent, the average effect of using $\tilde{C}^{uu}_n$ can be characterized by the its conditional average $\mE_n \tilde{C}_n^{uu}$, which is taken conditioned on the information of the current ensemble $\{u_n^{(j)}\}_{j=1}^J$. With some elementary probability, it can be derived that (see \cref{eq:ExpC})
\[
\mE_n \tilde{C}_n^{uu} = \lambda(1-\lambda) \diag(C_n^{uu}) + \lambda^2 C_n^{uu}\succeq \lambda(1-\lambda) \diag(C_n^{uu}), 
\]
where $\diag(C_n^{uu})$ is the diagonal part of $C_n^{uu}$, i.e. $ \diag(C_n^{uu}) (s,t) = \delta_{st} C_n^{uu} (s,t)$, where $\delta_{st}$ is the standard Kronecker delta. It is easy to see that $\diag(C_n^{uu})$ is full rank if the diagonal elements of $C_n^{uu}$ are strictly positive.

\subsection{DEKI with mean--deviation separation}
The convergence of naive DEKI is slow because there is a tradeoff between approximation and optimization: in order to approximate the nonlinear map $\mcH$ more accurately, we would need the ensemble  covariance collapse quickly; but if the ensemble collapse too quickly, the ensemble will stop moving based on \cref{eq:EKI0}. Moreover, analyzing the convergence of naive DEKI is also very challenging, since $\tilde{C}^{uu}_n$ is very difficult to track. 

An easy way to resolve this issue is to use separate dynamics for the mean and deviation update. To illustrate, note that \cref{eq:EKI0} can also be formulated as 
\begin{subequations}
\begin{align}
    \mean{u}_{n+1} =~& \mean{u}_{n} + h_n C_{n}^{uz} ( I + h_n C_{n}^{zz} )^{-1} ( z - \mean{\mcH(u_{n})} ), \label{eq:EKImean} \\
    \tau_{n+1}^{(j)} =~& \tau_{n}^{(j)} + h_n C_{n}^{uz} ( I + h_n C_{n}^{zz} )^{-1} ( \mean{\mcH(u_{n})} - \mcH(u_{n}^{(j)}) ). 
\label{eq:EKIdevi}
\end{align}
\end{subequations}
We will use different $h_n, C_n^{uz} $ and $ C_n^{zz}$ for the two separated dynamics. 

Since the ensemble mean will be used as the candidate solution, we hope the stationary point of its dynamics will also be a stationary point of $l(u)$. Therefore we will use the dropout covariances in \cref{eq:EKImean}. We will also use a step size $\tilde h_n=\tilde h\|C^{uu}_n\|^{-1}$, an adaptive choice adopted from \cite{MR3998631}. Note the factor $\|C^{uu}_n\|^{-1}$ is added to counter balance the fact that the ensemble will collapse, i.e. $C^{uu}_n\to 0$. This leads to
\begin{equation}
\label{eq:DEKImean} 
    \mean{u}_{n+1} = \mean{u}_{n} +  \tilde{h}_n \tilde{C}_{n}^{uz} ( I +  \tilde{h}_n \tilde{C}_{n}^{zz} )^{-1} ( z - \mcH(\mean{u}_{n}) ),
\end{equation}
where $\tilde{C}_{n}^{uz}, \tilde{C}_{n}^{zz}$ are generated by \cref{eq:dropCov}. 

As for the dynamics of ensemble deviation \cref{eq:EKIdevi}, we seek a modification where analysis of the ensemble covariance $C^{uu}_n$ is available. When $\mcH$ is a linear map, $C^{uu}_n$ is known to have  explicit Kalman type update formula \cite{MR3041539,MR3654885}. For this reason, we seek a linearization of \cref{eq:EKIdevi}, in the form of  
\begin{equation}
\label{eq:DEKIdevi} 
    \tau_{n+1}^{(j)} = \tau_{n}^{(j)} - h_n C_{n}^{uu}H_n\matT  ( I + h_n H_n C_{n}^{uu}H_n\matT  )^{-1} H_n \tau_n^{(j)}.
\end{equation}
We set the step size $h_n=h\|C^{uu}_n\|^{-1}$, where $h$ is a constant that can be smaller than $\tilde h$, so the ensemble will not collapse too fast.

In order for \cref{eq:DEKIdevi} to approximate \cref{eq:EKIdevi}, a natural choice will be 
\[
    H_n=C_n^{zu}(C_n^{uu})^\dagger, 
\]
where $(C_n^{uu})^\dagger$ is the pseudo-inverse of $C_n^{uu}$. Notably, \cref{eq:DEKIdevi} will then be identical to \cref{eq:EKIdevi} if $\mcH$ is linear. In our numerical test, this choice already works. But for theoretical analysis below, we would require $H_n$ to be uniformly bounded w.r.t. $n$ in $l_2$ norm, which may not hold for all nonlinear maps. On the other hand, we note that a truncated version is always available by linearzing the forward map:   
\begin{equation}
\label{eq:linearize}
    H_n = \begin{bmatrix}
        G_n  \\
        \mcC_0^{-1/2}
    \end{bmatrix} ,\quad G_n := \argmin_{\normo{G}_2\leq M_G} \norm{ C_n^{zu}(C_n^{uu})^\dagger - \begin{bmatrix}
        G  \\
        \mcC_0^{-1/2}
    \end{bmatrix} }_{\rm F}^2.
\end{equation}
Here $M_G$ is some big prefixed constant and $\norm{\cdot}_{\rm F}$ denotes the Frobenius norm. The optimization problem of \cref{eq:linearize} can be solved simply by truncation on the singular value, see details in \cref{app:linearize}.  

In summary, the {\bf dropout EKI (DEKI)} scheme is given by the mean update \cref{eq:DEKImean} using dropout ensemble covariance, and the deviation update \cref{eq:DEKIdevi} using a truncated linearization \cref{eq:linearize}. The algorithm is provided as \cref{alg:DEKI}. We note that while its formulation is more involved than EKI \cref{eq:EKI0}, DEKI's computational cost per iteration is only a constant multiple of the EKI. Our numerical analysis below will show that DEKI will obtain a solution of $\epsilon$ accuracy with $O(d_u \log\epsilon^{-1} /J)$ iterations when the problem is strongly convex. 

\begin{algorithm}
\caption{Dropout EKI}
\label{alg:DEKI}
\begin{algorithmic}[1]
\setstretch{1.15}
\STATE{Input dropout rate $1-\lambda$, reference step size $h,\tilde{h}$. Initialize $\{u_0^{(j)}\}_{j=1}^J, n=0$}.
\WHILE{not converge}
\STATE Update $n\gets n+1$.
\STATE{Generate $\rho_n(s) \topsm{\iid}{\sim} \text{Bernoulli}(\lambda)$ and compute dropout ensemble 
\[
\tilde{u}_n^{(j)} = \mean{u}_n + \tilde{\tau}_n^{(j)}, \quad \tilde{\tau}_n^{(j)} = \rho_n \circ \tau_n^{(j)}.
\] }
\ \vspace{-15pt}
\STATE{Compute the dropout covariance $\tilde{C}_n^{uz},\tilde{C}_n^{zz}$ by \cref{eq:dropCov}.}
\STATE{Find the linearized map $H_n$ by solving \cref{eq:linearize}.}
\STATE{Determine the step sizes $h_n = h \|C^{uu}_n\|^{-1}, \tilde{h}_n = \tilde{h} \|C^{uu}_n\|^{-1}$.}
\STATE{Update the mean and deviations separately by \cref{eq:DEKImean} and \cref{eq:DEKIdevi}.}
\ENDWHILE
\end{algorithmic}
\end{algorithm}

\begin{rems}
(1) We can also use $\tilde{C}^{uu}_n$ in \cref{eq:DEKIdevi}. We observe in numerical tests that as long as the ensemble is stable, using dropout in \cref{eq:DEKIdevi} behaves similarly as the DEKI scheme. However, this variant is difficult to analyze rigorously, so we present our current formulation.

(2) The step size $\tilde{h}_n = \tilde{h} \normo{C_n^{uu}}^{-1}$ offsets the ensemble collapse and ensures the fast convergence of EKI. Another popular choice is the constant step size $h_n \equiv h$, which is favorable when one wishes to allow noisy update scheme, or the smoothness constant is hard to estimate. But such choice leads to slower convergence \cite{MR3998631}, and with rate that might scale with the dimension $d_u$. One can also consider $h_n = h n^{-\alpha}$ used in \cite{MR4405495}, which leads to accelerated convergence. We mainly consider the first choice here, and the thorough comparison of different step size choices is out of scope of this paper.  
\end{rems}

\subsection{Related literature}

\subsubsection{EKI methodology}
EKI has been introduced in \cite{MR3041539} as an extension of the popular ensemble Kalman filter (EnKF) methodology to the treatment of inverse problems in the context of optimization. Both the EnKF and EKI have recently been surveyed in \cite{calvello22} primarily from a mean field perspective. However, practical algorithm use finite ensemble sizes and establishing convergence and computational efficiency for nonlinear problems have remained an open problem.


Here we focus first on the finite ensemble size effect and later also address the issue of computational efficiency. As discussed before, the subspace property will prevent the ensemble to converge to the optimal solution. However, some asymptotic behavior of EKI can be analyzed in linear \cite{MR3654885,MR3764752} and nonlinear \cite{MR4089506,MR4482475} cases. Roughly speaking, the ensemble converges to the optimal solution projected onto the subspace spanned by the initial ensemble. Localization \cite{MR4580673} is one method to overcome this issue, which will be discussed later in detail. Interested readers are referred to \cite{MR4405495,MR4482475} for further recent results on EKI.

EKI can also be understood as a zeroth order optimization method \cite{MR3041539,MR4089506,MR4405495}. In particular, it can be viewed as a gradient flow using the empirical covariance as a preconditioner \cite{MR3654885,MR4482475}. It approximates the gradients by the finite differences, making the method derivative-free. Among zeroth order optimization methods, it is known that there are methods with complexity $\mcO(d_u \log \varepsilon^{-1})$ to achieve $\varepsilon$ accuracy for strongly convex functions, such as the Nesterov random search method \cite{MR3627456}. The complexity of DEKI matches this result, and in fact we show that this rate is optimal even for linear inverse problems. 

\subsubsection{Dropout in deep learning}
Dropout technique is first introduced in neural network (NN) by Hinton et al. \cite{JMLR:v15:srivastava14a} as an implicit regularization technique to avoid overfitting. The idea is to randomly omit parts of the neurons in a NN during training. This method has proven to be very effective in improving the performance of NN in a wide range of tasks \cite{JMLR:v15:srivastava14a} and is accepted as a common practice now. There is a long literature exploring the practical implementation of different dropout methods in different types of NNs. Popular methods include the original dropout \cite{JMLR:v15:srivastava14a}, dropconnect \cite{pmlr-v28-wan13}, variational dropout \cite{NIPS2015_bc731692}, Monte Carlo dropout \cite{pmlr-v48-gal16} etc. 

Despite its rich application in deep learning, theoretic understanding of dropout is still limited. Most practitioners regard dropout as an implicit regularization method \cite{JMLR:v15:srivastava14a,pmlr-v28-wan13,10.5555/2999611.2999651}. By randomly dropping out neurons, dropout decreases the dependence on individual neurons to avoid overfitting. As a regularization method, dropout can also be viewed as a noisy perturbation of the output of the neurons, which encourages the model to learn more robust features. Another  explanation is to interpret dropout as a Bayesian approximation \cite{pmlr-v48-gal16}, where a probability distribution is imposed over the parameters, and dropout is used to draw samples of the parameters approximately. Then the ability to avoid overfitting can be justified by the use of a Bayesian model. 

The dropout method we use in DEKI is  similar to dropout NN only at a high level. Dropout in DEKI decreases the dependence on different parameters to avoid spurious long-distance correlation over the components, which is a well-known issue especially in the high dimensional and small ensemble settings. On the other hand, dropout in DEKI has different structures compared to that in the NN. One obvious difference is that dropout acts linearly in DEKI, while it acts highly nonlinearly in NN. Another difference is that dropout in DEKI is applied to the deviations, so that it preserves the structure of the parameters. But in NN, dropout might change the topology of the parameters. The exact reason and benefits of dropout for EKI should be analyzed independently. 

\subsubsection{Localized EKI}
Localization is a crucial technique used in EnKF to deal with the sampling errors due to a limited ensemble size \cite{1998MWRv..126..796H,2001MWRv..129..123H,2001MWRv129.7690,doi:10.3402/tellusa.v56i5.14462}. The idea is to artificially reduces the covariance between distant locations, typically by applying a distance-based weighting function. \cite{MR4580673} investigated the application of localization in EKI, and showed that the localized EKI (LEKI) scheme can be applied in the small ensemble settings and overcome the subspace issue of EKI effectively. 

Our scheme bears a resemblance to the LEKI, but avoids the potentially-involved design of
the localization methods, and is computationally cheaper since DEKI can be implemented using only matrix-vector multiplication while LEKI cannot. At a high level, DEKI can also be viewed as a stochastic implementation of the LEKI, see \cref{eq:ExpC} for details. We also emphasize that there are other possible ways to apply dropout in EKI, and here we just choose one version to illustrate the general idea and show theoretic properties.

\section{Analysis of DEKI}
\label{sec:analysis}
The analysis is divided into two parts. First we show the ensemble collapses in a controllable way by studying the dynamic of the spectrum of the ensemble covariance, from which we obtain lower bounds of the diagonal entries. Second we prove the convergence of DEKI by approximating it locally by a Gauss–Newton type scheme and derive contraction using the lower bounds obtained before. Here an important observation is that the dropout emphasizes the diagonal elements of the ensemble covariance in average, see \cref{eq:ExpC} for details.

\subsection{Assumptions} We make the following assumptions on the regularized forward operator $\mcH$ \cref{eq:RegProblem}, the linearized map $H_n$ \cref{eq:linearize} and the $l^2$-loss $l(u)$ \cref{eq:TikReg}.
\begin{asms}
\label{asm:main}
The following conditions hold:
\begin{enumerate}
    \item Bounded Hessian. $\mcH\in C^2$ and $\exists H>0$ s.t.
    \begin{equation}
    \label{eq:HessBound}
    \sup_u \left(\sum_{s=1}^{d_z} \norm{\nabla^2 \mcH_s(u)}_2^2 \right)^{1/2} \leq H,
    \end{equation}
    where $\mcH_s$ denotes the $s$-th component of $\mcH$. 
    \item Boundedness of $H_n$. $\exists 0<\gamma \leq M$ s.t. $\forall n\in\mN$,
    \begin{equation}
    \label{eq:LineHBound}
        \gamma^2 I \preceq H_n\matT  H_n \preceq M^2 I.
    \end{equation}
    \item $L$-smoothness of $l(u)$. $\exists L>0$, s.t. $\forall u,v \in\mR^{d_u} $,
    \begin{equation}
    \label{eq:Lsmooth}
        \normo{\nabla l(u) - \nabla l(v)} \leq L \norm{u-v}.
    \end{equation}
    \item Polyak-Łojasiewicz (PL) condition for $l(u)$. $\exists c>0$ s.t. $\forall u \in\mR^{d_u}$,
    \begin{equation}
    \label{eq:PL}
        \norm{\nabla l(u)}^2 \geq c (l(u) -l_{\min}),
    \end{equation}
    where $l_{\min} := \min_{u} l(u)$. 
\end{enumerate}
\end{asms}
\begin{rems}
\label{rem:asm_main}
(1) The bounded Hessian condition \cref{eq:HessBound} directly follows from that of the forward map $\mcG$. The boundedness of $H_n$ can be achieved when we take $H_n = [G_n\matT,\mcC_0^{-1/2}]\matT$ defined in \cref{eq:linearize}. Notice $  H_n\matT H_n = G_n\matT G_n + \mcC_0^{-1} $, since $\norm{G_n}\leq M_G$, suppose $\mcC_0^{-1} \succeq \gamma_0^2 I $, then \eqref{eq:LineHBound} holds with $\gamma = \gamma_0$ and $M = \sqrt{ M_G^2 + \norm{\mcC_0^{-1}} } $.

(2) The $L$-smoothness and PL condition for $l(u)$ are standard assumptions in optimization. Note PL condition is a weaker condition than strong convexity, and is widely adopted in the machine learning literature when lacking convexity \cite{MR4412181}.
\end{rems}

\subsection{Controllable ensemble collapse}
\label{sec:ensemble_ctrl}
The linearization \cref{eq:linearize} applied to the covariance evolution produces a behavior similar to that of a linear forward map. It allows us to have good controls on the ensemble collapse as follows:
\begin{proposition}
\label{prop:ensemble_collapse}
Consider the DEKI scheme for the regularized problem \cref{eq:RegProblem}. Assume that the linearized map $H_n$ satisfies $\gamma^2 I \preceq H_n\matT  H_n \preceq M^2 I$. Denote $C_n^{uu}$ as the empirical covariance. Choose adaptive step sizes $h_n= \theta \norm{C_n^{uu}}^{-1}$, where $\theta\leq M^{-2}$. Then the modified condition number of $C_n^{uu}$ is uniformly bounded:
\begin{equation}
\label{eq:KappaBound}
    \forall n\in\mN, \quad \kappa(C_n^{uu}) := \frac{\lambda_1(C_n^{uu})}{\lambda_r(C_n^{uu})} \leq \bar{\kappa} = \max \left\{ \kappa(C_0^{uu}), \frac{3M^2}{2\gamma^2} \right\}.
\end{equation}
where $\lambda_k(C_n^{uu})$ is the $k$-th largest eigenvalues of $C_n^{uu}$ and $r$ is the rank of $C_n^{uu}$, which is invariant under evolution: $
\rk (C_n^{uu}) \equiv r := \rk (C_0^{uu})$. Moreover, the ensemble collapses exponentially,
\begin{equation}
    \normo{C_{n}^{uu}} \leq \normo{C_0^{uu}} (1 + \gamma^2 \theta)^{-2n} .
\end{equation}
\end{proposition}
\begin{proof}
Rewrite the deviation update \cref{eq:DEKIdevi} as
\begin{equation*}
\label{eq:pfDeviLinear}
\begin{split}
    \tau_{n+1}^{(j)} =~& \Rectbrac{ I - h_n C_{n}^{uu} H_n\matT  ( I + h_n H_n C_{n}^{uu} H_n\matT  )^{-1} H_n } \tau_{n}^{(j)} \\
    =~& ( I + h_n C_{n}^{uu} H_n\matT  H_n )^{-1} \tau_{n}^{(j)},
\end{split}
\end{equation*}
where we use the Sherman-Morrison-Woodbury formula (see \cref{lem:SMW}) by taking $A= I, U = h_n C_{n}^{uu} H_n\matT  $ and $V = H_n\matT  $.

Recall by definition $ C_n^{uu} = \frac{1}{J-1} \sum_{j=1}^J \tau_{n}^{(j)} (\tau_{n}^{(j)})\matT $. Therefore,
\[
    C_{n+1}^{uu} = ( I + h_n C_{n}^{uu} H_n\matT  H_n )^{-1} C_{n}^{uu} ( I + h_n H_n\matT  H_n C_{n}^{uu} )^{-1}.
\]
By assumption, $\gamma^2 I \preceq H_n\matT H_n \preceq  M^2 I.$ Apply \cref{lem:Ccmp} to $C_{n+1}^{uu}$, we obtain
\begin{equation}
\label{eq:pfCBound}
    C_n^{uu} ( I + M^{2} h_n C_n^{uu} )^{-2} \preceq C_{n+1}^{uu} \preceq C_n^{uu} ( I + \gamma^{2} h_n C_n^{uu} )^{-2}.
\end{equation}
Notice when $0\leq x\leq a^{-1}$, 
\[
    \frac{x}{1+3ax} \leq \frac{x}{(1+a x)^2} \leq \frac{x}{1+2ax}.
\]
Applying \cref{lem:PSDcmp} (3), when $h_n = \theta \normo{C_n^{uu}}^{-1}$ and $\theta \leq M^{-2} \leq \gamma^{-2}$, it holds
\begin{equation*}
    C_n^{uu} ( I + 3 M^{2} h_n C_n^{uu} )^{-1} \preceq C_{n+1}^{uu} \preceq C_n^{uu} ( I + 2 \gamma^{2} h_n C_n^{uu} )^{-1}.
\end{equation*}
Denote $\lambda_k(C_n^{uu})$ as the $k$-th largest eigenvalues of $C_n^{uu}$, then (see \cref{lem:PSDcmp} (2))
\begin{equation}
\label{eq:pfCEigenBound}
    \frac{\lambda_k(C_n^{uu})}{ 1+ 3 h_n M^2 \lambda_k(C_n^{uu}) } \leq \lambda_k(C_{n+1}^{uu}) \leq \frac{\lambda_k(C_n^{uu})}{1+2h_n\gamma^2 \lambda_k(C_n^{uu})}.
\end{equation}
Hence $ \lambda_k(C_n^{uu}) > 0 \ioi \lambda_k(C_{n+1}^{uu})>0$, which implies that the rank of $C_n^{uu}$ remains constant $r= \rk(C_0^{uu})$. From \cref{eq:pfCEigenBound}, we get for $1\leq k\leq r$,
\begin{equation*}
    2 h_n \gamma^2 \leq \lambda_k^{-1}(C_{n+1}^{uu}) -  \lambda_k^{-1}(C_n^{uu}) \leq 3 h_n M^2.
\end{equation*}
Denote the modified condition number $\kappa_n = \lambda_1(C_n^{uu})/\lambda_r(C_n^{uu})$. Then
\begin{equation*}
    \begin{split}
        \kappa_{n+1} =~& \frac{\lambda_{r}^{-1}(C_{n+1}^{uu})}{\lambda_{1}^{-1}(C_{n+1}^{uu})} \leq \frac{\lambda_{r}^{-1}(C_{n}^{uu}) + 3 h_n M^2 }{\lambda_{1}^{-1}(C_{n}^{uu}) + 2 h_n \gamma^2 } \\
        \leq~& \max\left\{ \frac{\lambda_{r}^{-1}(C_{n}^{uu})}{\lambda_{1}^{-1}(C_{n}^{uu})}, \frac{3 M^2 }{2 \gamma^2 } \right\} = \max \{ \kappa_n, 3 M^2/ 2\gamma^2 \}.
    \end{split}
\end{equation*}
By induction, we deduce that
\begin{equation*}
    \kappa_n \leq \bar{\kappa} := \max\{ \kappa_0, 3 M^2/2\gamma^2 \}.
\end{equation*}
Finally, applying \cref{lem:PSDcmp} (1) to \cref{eq:pfCBound}, we obtain 
\begin{equation*}
\begin{split}
    \normo{C_{n+1}^{uu}} \leq \lambda_{\max} \Big( &C_n^{uu} ( I + \gamma^2 h_n C_n^{uu} )^{-2} \Big) =  \frac{\normo{C_{n}^{uu}}}{(1 + \gamma^2 h_n \normo{C_{n}^{uu}})^2}. \\
    \St~& \normo{C_{n}^{uu}} \leq \normo{C_0^{uu}} (1 + \gamma^2 \theta)^{-2n}.
\end{split}
\end{equation*}
Note the equality above holds since $\psi(x) = x(1+\gamma^2 h_n x)^{-2}$ is monotone increasing on $x\in[0,\normo{C_n^{uu}}]$, due to the fact that when $h_n = \theta \normo{C_n^{uu}}^{-1}$ and $\theta \leq M^{-2} \leq \gamma^{-2}$,
\[
    \gamma^2 h_n x \leq \gamma^2 h_n \normo{C_n^{uu}} \leq 1 \St \psi'(x) = \frac{1-\gamma^2 h_n x}{(1+\gamma^2 h_n x)^{2}} \geq 0.
\]
\end{proof}

The convergence analysis also requires the control of the diagonal elements of the covariance matrix, whose proof is mainly based on the invariance of the column space of the covariance matrix. We prove
\begin{proposition}
\label{prop:ensemble_collapse_diag}
Under the settings of \cref{prop:ensemble_collapse}. The lower bound for the diagonal elements of $C_n^{uu}$ holds:
\begin{equation}
    \begin{split}
        \min_s C_n^{uu} &(s,s) \geq \bar{\kappa}^{-1} \norm{C_n^{uu}} \min_s P(s,s),
    \end{split}
\end{equation}
where $\bar{\kappa}$ is the upper bound in \cref{eq:KappaBound}, and $P\in\mR^{d_u\times d_u}$ is the the orthogonal projector onto the column space $ \mcV_n = \text{Im}(C_n^{uu}) $. The projector is invariant under evolution since  $\mcV_n $ is invariant: $\mcV_n\equiv \mcV = \text{Im}(C_0^{uu})$.
\end{proposition}
\begin{rem}
To obtain a meaningful lower bound, we should assume
\begin{equation}
\label{eq:PBound}
    \min_s P(s,s)>0.
\end{equation}
It is equivalent to $\min_s C_0^{uu}(s,s)> 0$, which is necessary since the dropout procedure cannot explore directions with zero covariance. This condition can be achieved easily by using Gaussian initialization.
\end{rem}
\begin{proof}
First we prove that the column space of $C_n^{uu}$ is invariant. Note $\text{Im}(C_n^{uu}) = \text{span}\{ \tau_n^{(j)}:j=1,\dots J\} $. Observe that for any $v\in\mR^{d_u}$ that satisfies $\forall 1\leq j\leq J, v\matT  \tau_n^{(j)} = 0$, it holds
\[
    v\matT  \tau_{n+1}^{(k)} = v\matT  \Brac{ \tau_{n}^{(k)} + \frac{h_n}{J-1} \sum_{j=1}^J \tau_n^{(j)} ( H_n(u_{n}^{(j)}) - \mean{H_n(u_{n})} )\matT  w_n^{(j)} } = 0,
\]
where we denote $w_n^{(j)} = ( I + h_n H_n C_{n}^{uu} H_n\matT  )^{-1} ( \mean{H_n(u_{n})} - H_n(u_{n}^{(j)})) $. Thus by induction, we obtain $ \forall n, 1\leq j\leq J, v\matT  \tau_n^{(j)} = 0 $. This implies that
\[
    \text{Im}(C_n^{uu}) \subset \text{Im}(C_0^{uu}).
\]
While by \cref{prop:ensemble_collapse}, $\rk(C_n^{uu}) = \rk(C_0^{uu})$, we conclude that $\text{Im}(C_n^{uu}) = \text{Im}(C_0^{uu})$.

To control the diagonal element of $C_n^{uu}$, consider the singular value decomposition $ C_n^{uu} = Q_n \Lambda_n Q_n\matT  $, then
\[
    C_n^{uu}(s,s) = \sum_{t=1}^r \lambda_t(C_n^{uu}) (Q_n(s,t))^2.
\]
Notice $P = Q_nQ_n\matT $ is the orthogonal projector onto the invariant colunm space $\mcV = \text{Im} (C_0^{uu})$, and therefore is invariant. Then we can control the lower bound of the diagonal element as follows:
\[
    \min_s C_n^{uu}(s,s) \geq \lambda_r(C_n^{uu}) \min_s \sum_{t=1}^r (Q_n(s,t))^2 \geq \bar{\kappa}^{-1} \normo{C_n^{uu}} \min_s P(s,s).
\]
\end{proof}

\subsection{Convergence to optimal solution}
The convergence of DEKI is mainly guaranteed by the positive definiteness of average dropout ensemble covariance, this is similar to the localization effect used in LEKI. In particular, we claim that
\begin{equation}
\label{eq:ExpC}
\bar{C}_n^{uu} = \mE_n \tilde{C}_n^{uu} = \lambda(1-\lambda) \diag(C_n^{uu}) + \lambda^2 C_n^{uu}, 
\end{equation}
where $\diag(C_n^{uu})$ is the diagonal part of $C_n^{uu}$, i.e. $ \diag(C_n^{uu}) (s,t) = \delta_{st} C_n^{uu} (s,t)$, and here $\delta_{st}$ is the standard kronecker delta. If we denote $\Psi = \lambda^2 I + \lambda(1-\lambda) E$, where $E(s,t)\equiv 1$, then $\bar{C}_n^{uu} = \Psi \circ C_n^{uu}$.

To prove \cref{eq:ExpC}, just notice when $s\neq t$,
\begin{equation*}
\begin{split}
    \bar{C}_n^{uu}(s,t) =~& \mE_n \Brac{ \frac{1}{J-1} \sum_{j=1}^J \rho_n(s) \rho_n(t) \tau_n^{(j)}(s) \tau_n^{(j)}(t) } \\
    =~& \frac{1}{J-1} \sum_{j=1}^J \lambda^2 \tau_n^{(j)}(s) \tau_n^{(j)}(t) = \lambda^2 C_n^{uu}(s,t).
\end{split}
\end{equation*}
And when $s=t$,
\begin{equation*}
\begin{split}
    \bar{C}_n^{uu}(s,s) =~& \mE_n \Brac{ \frac{1}{J-1} \sum_{j=1}^J \rho_n^2(s) (\tau_n^{(j)}(s))^2 } \\
    =~& \frac{1}{J-1} \sum_{j=1}^J \lambda (\tau_n^{(j)}(s))^2 = \lambda C_n^{uu}(s,s).
\end{split}
\end{equation*}
The convergence analysis is based on the connection of EKI with Gauss-Newton. To be specific, notice the mean evolution \cref{eq:DEKImean} in DEKI can be approximated by the following Gauss-Newton type iteration
\begin{equation}
\label{eq:pfGN}
    \mean{u}_{n+1}' = \mean{u}_n + \tilde{h}_n \tilde{C}_{n}^{uu} G_n\matT  ( I + \tilde{h}_n G_n \tilde{C}_{n}^{uu} G_n\matT  )^{-1} ( z - \mcH(\mean{u}_{n}) ),
\end{equation}
where $G_n = \nabla \mcH(\mean{u}_n)$ is the Jacobian at $\mean{u}_n$. Note it is the Gauss-Newton update for the loss function
\[
    l_{n}(u) = \normo{u-\mean{u}_n}_{\tilde{C}_n^{uu}}^2 + \tilde{h}_n \normo{\mcH(u)-z}^2.
\]
We can bound the difference of the two updates by the following lemma. It takes a similar form as Proposition 3.3 in \cite{MR4405495}, but the upper bound depends on $d_u$ there. Here we improve the bounds so that there is no dimension dependence.
\begin{proposition}
\label{prop:LinearErr}
Under the settings of \cref{thm:convergence}, it holds
\begin{equation}
\label{eq:LinearErr}
    \normo{\mean{u}_{n+1} - \mean{u}_{n+1}' } \leq C \tilde{h}_n \normo{\tilde{C}_n^{uu}}^{3/2} \normo{ z - \mcH(\mean{u}_{n}) },
\end{equation}
where $C = (J-1)^{3/2} H$ is a constant independent of $d_u$ and $H$ is the bound of the Hessian of $\mcH$ defined in \cref{eq:HessBound}.
\end{proposition}

The proof is postponed to \cref{app:linerr}. 
Notice though the ensemble evolves stochastically, \cref{prop:ensemble_collapse}, \cref{prop:ensemble_collapse_diag} and \cref{prop:LinearErr} hold deterministically once the intiailization is fixed. This mainly comes from
the fact that we do not use the dropout covariance matrices in the evolution of the
ensemble deviation.
The proposition shows that DEKI can be approximated well by the Gauss-Newton update as the ensemble collapses. It is expected to obtain exponential convergence for such scheme under the PL condition \cref{eq:PL}. To be specific, we prove the main theorem as follows.
\begin{theorem}
\label{thm:convergence}
Consider the DEKI scheme for the regularized problem \cref{eq:RegProblem} in the high dimensional regime $d_u\gg J$. Assume \cref{asm:main} holds, and use standard Gaussian initialization so that \cref{eq:PBound} holds. Choose adaptive step sizes 
\begin{equation}    \label{eq:Ada_step}
    h_n = \theta \normo{C_n^{uu}}^{-1},\quad \tilde{h}_n = \mu \normo{C_n^{uu}}^{-1},
\end{equation}
where $\theta \leq M^{-2}, \mu \leq L^{-1}$, and $L, M$ are defined in \cref{asm:main}. Then there exists $n_0 \in \mN$ s.t. when $n>n_0$, the $l^2$-loss \cref{eq:TikReg} converges exponentially,
\begin{equation}
    \mE_{n_0} l(\mean{u}_{n}) - l_{\min} \leq (1-\beta)^{n-n_0} (l(\mean{u}_{n_0}) - l_{\min} + C_2),
\end{equation}
where $\beta \in (0,1)$ and $C_2$ are some constants that can be determined explicitly.
\end{theorem}
\begin{rems}
(1) The convergence rate $\beta$ admits explicitly expression \cref{eq:beta} where $\beta_0$ in defined in \cref{eq:beta0}. Under mild scaling assumptions, one can show that $\beta = \mcO(J/d_u)$ (see \cref{lem:scaling}). Therefore, to reach $\varepsilon$-accuracy, the needed steps are
\begin{equation}
    N_\varepsilon = \mcO \Brac{ d_u \log \varepsilon^{-1} / J }.
\end{equation}
Such linear scaling w.r.t. the dimension $d_u$ is optimal. For more details, please see \cref{prop:complexity}.

(2) Here $n_0$ is the warm-up steps and can be taken as some fixed number depending only on the initialization, or in other words, $n_0$ is a $\mcF_0$-measurable random variable. One can show that $n_0 = \mcO(\log d_u)$ using the scaling arguments in \cref{lem:scaling}. Note $\mcO(\log d_u)$ is usually viewed as $\mcO(1)$ quantity, and the term $l(\mean{u}_{n_0})$ is harmless to our result. From \cref{lem:meanstab}, we can see that the $l^2$-loss can only have constant growth: $l(\mean{u}_{n_0}) \leq C' l(\mean{u}_{0})$.

(3) Computational costs comparison. For vanilla EKI, the ensemble size should be at least $J=\mcO(d_u)$, and it takes $\mcO(\log \varepsilon^{-1})$ steps to reach $\varepsilon$-accuracy. In each step, it takes $J$ evaluations of the forward map, and $\mcO(J^2 (d_u+d_y))$ additional costs due to the matrix-vector multiplications. Thus the overall costs would be $\mcO(d_u\log\varepsilon^{-1})$ forward evaluations and $\mcO(d_u^2 (d_u+d_y) \log \varepsilon^{-1}))$ additional costs. For DEKI, it takes $\mcO(J^{-1}d_u\log \varepsilon^{-1})$ steps when $\varepsilon$ is small, and thus the overall costs will be $\mcO(d_u\log \varepsilon^{-1})$ forward evaluations, and $\mcO(J d_u (d_u+d_y) \log \varepsilon^{-1} )$ additional costs. Since $J\ll d_u$, we can see that DEKI reduces the additional cost while using the same order of forward evaluations. 

(4) We can also obtain almost sure convergence with the same linear rate $\beta$ using submartingale convergence theorem, see \cite{weissmann2024almost}. Denote $X_n = (1+\beta)^n \Brac{ l(\mean{u}_n) - l_{\min}}$ and $\varepsilon_n = C_1 (1+\beta)^{n+1} \delta^n l_{\min}$, the one-step analysis \cref{eq:pf_onestep} shows that when $n\geq n_0$, 
\[
    \mE_n X_{n+1} \leq (1-\beta^2) X_n + \varepsilon_n. 
\]
The supermartingale convergence theorem implies $X_n \gto 0$ a.s., which implies the almost sure linear convergence.
\end{rems}

The proof of \cref{thm:convergence} is based on the following one-step analysis. 
\begin{lemma}
\label{lem:onestep}
Under the settings of \cref{thm:convergence}, it holds
\begin{equation}    \label{eq:onestep}
    \mE_n l(\mean{u}_{n+1}) - l_{\min} \leq (1- 2\beta_0 + \Delta_n) (l(\mean{u}_n) - l_{\min}) + \Delta_n l_{\min},
\end{equation}
where $\beta_0,\Delta_n$ are determined by
\begin{equation}    \label{eq:beta0}
    \beta_0 := c \lambda(1-\lambda) \bar{\kappa}^{-1} \min_s P(s,s) \cdot \frac{\mu(1+ 2\mu M^2)}{4(1+\mu M^2)^2}.
\end{equation}
\begin{equation}    \label{eq:Delta_n}
    \Delta_n := 2C \mu ( M + L \mu^{1/2} ) \normo{C_n^{uu}}^{1/2} + L C^2 \mu^2 \normo{C_n^{uu}}.
\end{equation}
Here $c,M,L$ are defined in \cref{asm:main}, $C$ is the contant in \cref{prop:LinearErr}, $\bar{\kappa}$ is the upper bound in \cref{eq:KappaBound}, and $P\in\mR^{d_u\times d_u}$ is the orthogonal projector onto $\mcV = $ Im$(C_0^{uu})$ in \cref{prop:ensemble_collapse_diag}. 
\end{lemma}
See \cref{app:onestep} for the proof. 

\begin{proof}[Proof for \cref{thm:convergence}]
By \cref{lem:onestep}, the one-step estimation \cref{eq:onestep} holds. First we control the last term $\Delta_n l_{\min}$ in \cref{eq:onestep}. By \cref{prop:ensemble_collapse}, the ensemble collapses exponentially,
\[
    \normo{C_n^{uu}} \leq \normo{C_0^{uu}} \delta^{2n}, \quad \delta := (1+\gamma^2\theta)^{-1} \in (0,1).
\]
Thus $\Delta_n$ can be bounded by
\[
    \Delta_n \leq C_1 \delta^{n},
\]
for $C_1 = 2 C \mu (M + L \mu^{1/2} ) \normo{C_0^{uu}}^{1/2} + L C^2 \mu^2 \normo{C_0^{uu}}$ (see \cref{eq:Delta_n}). Consider finding $n_0$ so that $\forall n\geq n_0, \Delta_n \leq \beta_0$. One can take
\begin{equation}    \label{eq:n0}
    n_0 = \Big\lceil \frac{\log (\beta_0^{-1}C_1) }{\log \delta^{-1}} \Big\rceil.
\end{equation}
So that when $n>n_0$, $1-2\beta_0 + \Delta_n \leq 1- \beta_0$, and \cref{eq:onestep} leads to 
\begin{equation}    \label{eq:pf_onestep}
    \mE_n l(\mean{u}_{n+1}) - l_{\min} \leq (1-\beta_0) (l(\mean{u}_n) - l_{\min}) + C_1 \delta^{n} l_{\min}.
\end{equation}
Next we take 
\begin{equation}    \label{eq:beta}
    \beta = \min \{ \beta_0, \frac{\gamma^2 \theta}{ 2 ( 1 + \gamma^2 \theta ) } \},    
\end{equation}
then \cref{eq:pf_onestep} stills holds with $\beta_0$ replaced by $\beta$. By induction, one obtain for $n>n_0$, 
\begin{align*}
    & \mE_{n_0} l(\mean{u}_{n}) - l_{\min} \\
    \leq~& (1-\beta)^{n-n_0} ( l(\mean{u}_{n_0}) - l_{\min}) + C_1 l_{\min} \sum_{k=n_0}^{n-1} (1-\beta)^{n-1-k} \delta^{k}.
\end{align*}
By definition, $\beta \leq  \frac{\gamma^2 \theta}{ 2 ( 1 + \gamma^2 \theta ) } = \frac{1}{2} (1-\delta) \St 1 - \beta \geq \frac{1}{2} (1+\delta) > \delta $, so that
\[
    \sum_{k=n_0}^{n-1} (1-\beta)^{n-1-k} \delta^{k} = (1-\beta)^{n-1} \sum_{k=n_0}^{n-1} \Brac{\frac{\delta}{1-\beta}}^k  \leq \frac{(1-\beta)^{n}}{1-\beta - \delta}.
\]
So that one obtains
\[
    \mE_{n_0} l(\mean{u}_{n}) - l_{\min} \leq (1-\beta)^{n-n_0} (l(\mean{u}_{n_0}) - l_{\min} + C_2 ),
\]
where $C_2 = (1-\beta)^{n_0} C_1/(1-\beta - \delta)$. 
\end{proof}

\begin{lemma}   \label{lem:scaling}
Consider the high dimensional regime $J\ll d_u$. Assume that the $\mean{\kappa} = \mcO(1)$ and $c = \mcO(M^2)$, where $\mean{\kappa}$ is defined in \cref{eq:KappaBound} and $c,M$ are defined in \cref{asm:main}. Take the step size $\theta \sim M^{-2}$ (see \cref{eq:Ada_step}), then for the convergence rate \cref{eq:beta0} and the warm-up steps \cref{eq:n0}, it holds that 
\[
    \beta = \mcO(J/d_u), \quad n_0 = \mcO(\log d_u).
\]
Here $\mcO$ hides the constants that are independent of $d_u$.
\end{lemma}
The proof is postponed to \cref{app:scaling}.

\begin{rem}    \label{rem:scaling}
We justify the two assumptions $\mean{\kappa} = \mcO(1)$ and $c = \mcO(M^2)$ in \cref{lem:scaling}. 

(1) By definition, $\mean{\kappa} = \max\{ \kappa(C_0^{uu}), \frac{3M^2}{2\gamma^2}\}$. First note $\frac{M^2}{\gamma^2}$ is approximately the condition number of the regularized map $\mcH$. It is reasonable to assume that it is dimension independent, or otherwise the problem itself is  ill-posed.  For $\kappa(C_0^{uu})$, by definition it is the ratio of the largest and smallest nonzero eigenvalues of $C_0^{uu}$. In the high dimensional regime where $J\ll d_u$, using Gaussian initialization would make $C_0^{uu}$ close to the projection $P$, and thus $\kappa(C_0^{uu}) = \mcO(1)$. 

(2) Note $M^2 \sim \gamma^2$ as we assume the condition number $M^2/\gamma^2 = \mcO(1)$. It is also natrual to assume that $c \sim \gamma^2$, as they are both the lower bounds of the regularized map $\mcH$. When $\mcH$ is linear, one can even take $c = \gamma^2$. Since $\nabla^2 l = \mcH\matT \mcH$, the Polyak-Łojasiewic number can be taken as $c = \lambda_{\min} (\mcH\matT \mcH) = \gamma^2$. Finally, $c = \mcO(M^2)$ follows easily from $M^2\sim\gamma^2$ and $c\sim \gamma^2$. 
\end{rem}

\subsection{Query complexity analysis}
In this section, we consider the lower bound of query complexity for zeroth order algorithms to solve inverse problems, which is closely related to the optimization problems. In the optimization literature, zeroth order algorithms are those who can only evaluate the loss function $l$ at a query point, but not the derivatives of $l$. The query complexity is the number of query points needed for the algorithm to reach a certain accuracy. We mention that since the first formal study in \cite{MR702836}, a lot of results are obtained for different types of optimization algorithms and problem classes, for instance \cite{MR2142598,MR4163541,NIPS2009_2387337b}.

It is natural to extend to the inverse problem setting for zeroth order algorithms that can only evaluate $\mathcal{G}(u)$ at a query point $u$. Ensemble Kalman type algorithms clearly are zeroth order algorithms, and their query complexity will be $nJ$, where $n$ is iteration number needed and $J$ is the ensemble size. Query complexity is important because the evaluation of $\mathcal{G}$, which is often a simulation of an expensive black box model, is the main cost of the operation.
\cref{thm:convergence} indicates the query complexity of DEKI is $\mcO(d_u \log \varepsilon^{-1})$. Next, we show such linear dependence on $d_u$ is optimal.

Mathematically speaking,  a zeroth order algorithm $A$ solving \cref{eq:InverseProblem} can be described as a map of the form
\[
    A: (U,Y,y) \mapsto u.
\]
Here $U = (u_1,\dots,u_n) \in \mR^{d_u\times n}$ are the query points, $Y = (y_1,\dots,y_n) \in \mR^{d_y\times n}$ are the model outputs (i.e. $y_i = \mcG(u_i)$), $y\in\mR^{d_y}$ is the observation data that we try to invert, and $u\in\mR^{d_u}$ is the output of the algorithm. We say $A(U,Y,y)$ is a zeroth order algorithm with $n$ \textit{admissible} query, if the query points $u_1,\ldots,u_n$ are obtained from some admissible rule, i.e. $u_{i+1}=B(y,u_t,\mathcal{G}(u_t),t\leq i)$ for some update rule $B$, and thus the output is a map $A(U,\mathcal{G}(U),y)$. 

Next we show that even for linear inverse problems, with $n\leq d_u/2$ query points, any algorithms cannot identify the optimal solution to \cref{eq:TikReg} accurately.
\begin{proposition}
\label{prop:complexity}
Consider solving \cref{eq:TikReg} for linear forward maps with regularization operator $\mcC_0^{-1/2}=I$. Then for any zeroth order algorithm $A$ using any $n$ admissible query $U\in \mR^{d_u\times n}$ where $n\leq d_u/2$, it holds that
\[
\max_{G:\|G\|\leq 2} \|A(U,GU,y)-u^*(G)\|>0.1\norm{y},
\]
where $u^*(G)$ is the optimal solution to \cref{eq:TikReg}.
\end{proposition}
Proofs can be found in \cref{app:complexity}. The proposition implies that we need at least $n>d_u/2$ queries of the forward map to ensure the algorithm $A$ finds the optimal solution within $\mcO(1)$ error. For DEKI, in each step we use $2J$ queries, so that at least $\mcO(d_u/J)$ steps are needed. Also note that DEKI reaches the optimal $\varepsilon$-dependence ($\mcO(\log \varepsilon^{-1})$) as a zeroth order stochastic algorithm, see \cite{MR2142598,MR3627456}. This proves the optimality of DEKI.

\section{Numerical examples}
\label{sec:numerics}
We will apply DEKI to two numerical examples with comparison to vanilla EKI and LEKI. For fair comparison, we use the same step size of these EKI type algorithms.

\subsection{Linear transport equation}  \label{sec:LTE}
Consider the inverse problem of determining the initial wave based on observations on a later wave field. Here we consider the toy model where the wave speed is constant. The one-dimensional linear transport equation is
\[
    \pdiff{U}{t}(x,t) + a \pdiff{U}{x}(x,t) = 0, \quad x \in [0,1].
\]
For simplicity, we use the periodic boundary condition $U(t,0) = U(t,1)$. Then the analytical solution is $U(x,t) = U( \{ x-at \} ,0)$, where $\{x\} = x - \lfloor x \rfloor$ is the fractional part of $x$. The analytical solution is used in the numerical tests.

Suppose we can observe the wave at time $T=1$ at $d_y$ equi-spaced points $x^D_i = i/d_y ~(i=1,\dots d_y)$, i.e. the data is generated by
\begin{equation}
\label{eq:LTEobs}
    y = (y_1,\dots,y_{d_y}), \quad y_i = U(T,x^D_i) + \xi_i,
\end{equation}
where $\xi_i\sim \mcN(0,\sigma^2)$ and $\sigma = 10^{-2}$ is the observational noise. Introduce a computational grid with $d_u$ equally spaced grid points $x_i = i/d_u$ and denote the restriction of spatial functions $U(x)$ to this grid by
\[
    u = (U(x_1),\dots U(x_{d_u})) \in \mR^{d_u}.
\]
The forward map is then defined by
\begin{equation}
    \mcG: u_0 \in \mR^{d_u} \mapsto y \in \mR^{d_y},
\end{equation}
where $ u_0 = (u(0,x_1),\dots,u(0,x_{d_u})) \in \mR^{d_u} $ is the discrete initial wave to be inferred from the observations \cref{eq:LTEobs}. 

In numerical experiments, $u_0$ is randomly generated $u_0 \overset{\iid} \sim \mcN(0,I_{d_u})$, and we choose the regularization operator $\mcC_0^{-1/2} = \gamma d_u^{-1/2} I$ where $\gamma = 0.1$. We will run the DEKI \cref{alg:DEKI} with dropout rate $\lambda = 0.5$. Fix reference step size $\tilde{h}=2.5$ and set the ratio $h_n/\tilde{h}_n=0.1$, i.e.,
\begin{equation}
\label{eq:StepSizes}
    \tilde{h}_n = \frac{\tilde{h}}{\norm{C_n^{uu}} + \epsilon_0}, \quad h_n = 0.1 \tilde{h}_n.
\end{equation}
Note we add a small parameter $\epsilon_0=10^{-12}$ for numerical stability. The ensemble is initialized as Gaussian variables $u_0^{(j)} \overset{\iid} \sim \mcN(0,\gamma^2 I_{d_u})$. 

We compare DEKI with vanilla EKI and LEKI. For each experiment, we run $N=10^2$ steps and plot the relative loss
\begin{equation}
\label{eq:RelaMisfit}
    e_n = \frac{l(\mean{u}_n) - l_{\min}}{\norm{y}^2},
\end{equation}
where $l_{\min}$ is obtained by optimizing the loss function. We repeat the experiments $N_{rep} = 100$ times and record the mean and deviation of the relative loss. For each experiment, the data $y$ is fixed and different realizations of the initial ensemble and dropout are used.

The comparison is shown in \cref{fig:LTEcmpMisfit}. We can see that in the high dimension and small ensemble regime, the vanilla EKI fails due to the subspace issue, while DEKI and LEKI work well. We can also see the linear convergence of DEKI. 

\begin{figure}[htbp]
    \centering
    \includegraphics[scale = 0.33, trim = 50 200 50 200]{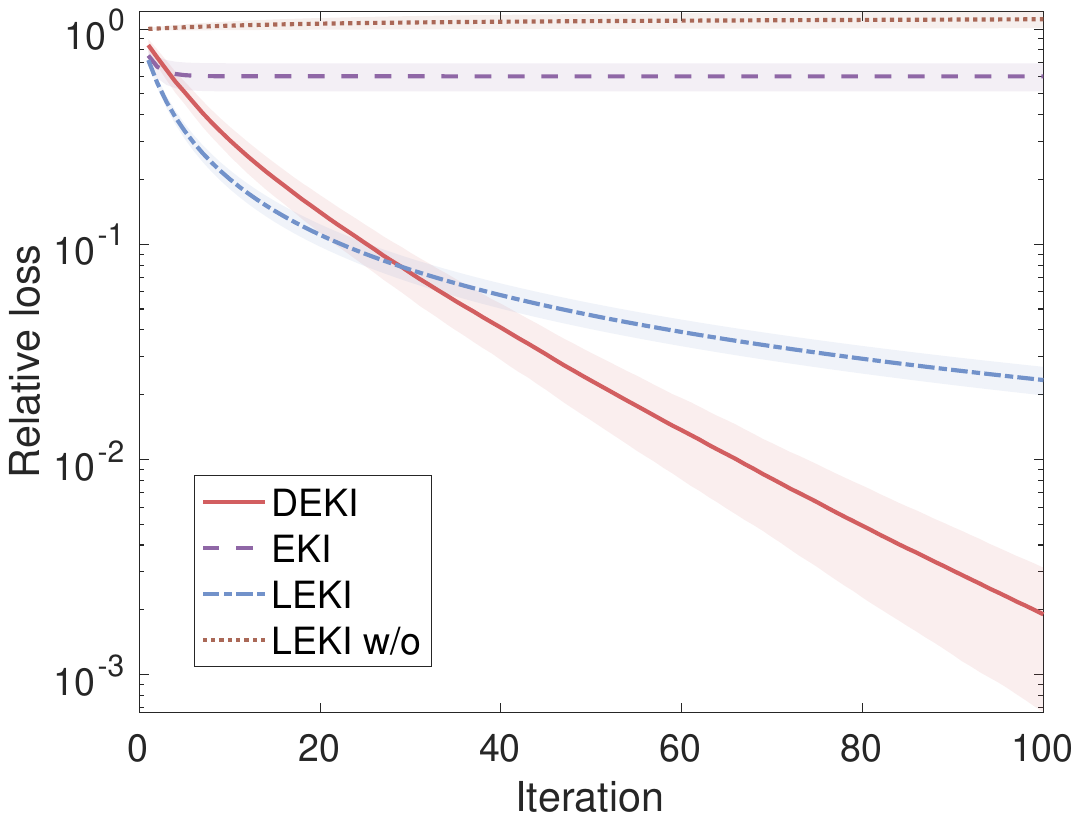}
    \includegraphics[scale = 0.33, trim = 0 200 50 200]{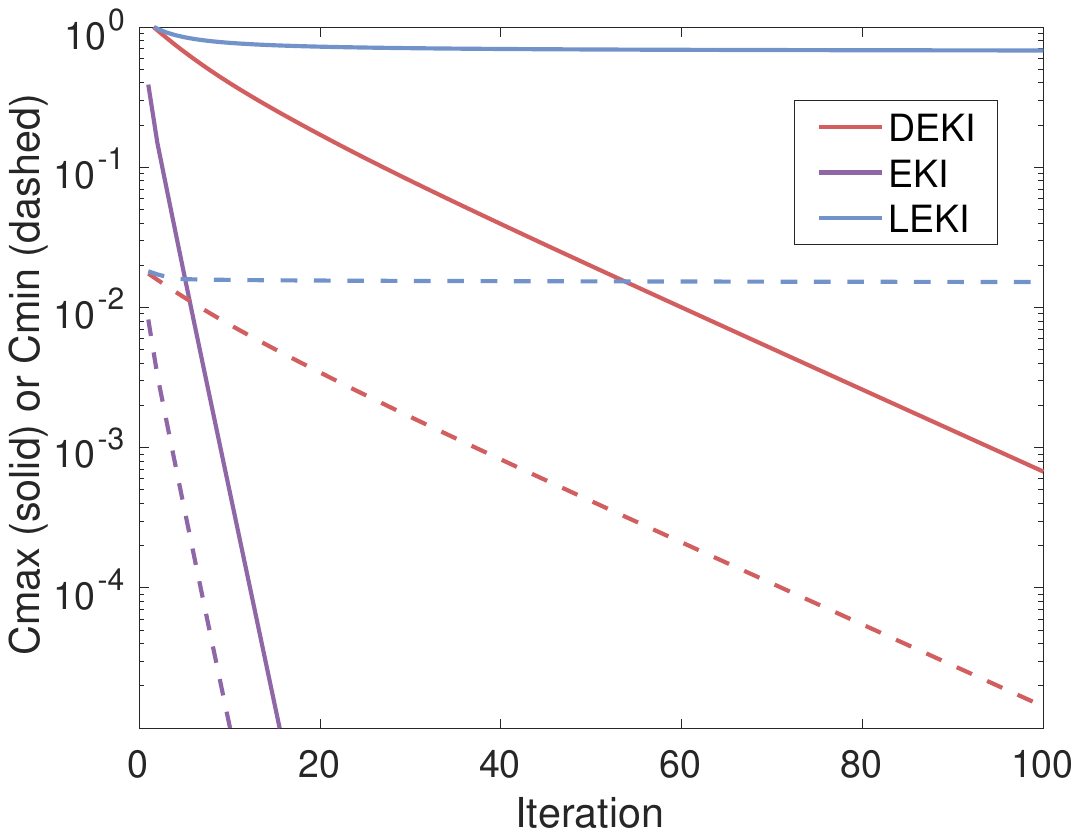}
    \caption{Comparison of DEKI with vanilla EKI and LEKI in linear transport equation model with $d_u = 120, J = 20$. Left: decay of the relative data misfit \cref{eq:RelaMisfit}, here LEKI w/o stands for LEKI without prior knowledge of the transport speed. Right: collapse of the ensemble, the solid line represents $\max_{s} C_n^{uu}(s,s)$ and the dashed line represents $\min_{s} C_n^{uu}(s,s)$.}
    \label{fig:LTEcmpMisfit}
\end{figure}

As discussed before, LEKI needs to design the localization matrix, which requires prior knowledge (here it is the transport speed). Details of the LEKI scheme and the localization design are provided in \cref{app:LEKI}. We show that without such prior knowledge, the LEKI still fails if the localization is not well designed. We also show the controllable decay of the maximum and minimum of the diagonal elements of the covariance matrix, confirming our analysis result.

We compute the convergence rate by
\begin{equation}
\label{eq:ConvRate}
    r = \frac{ \log e_m - \log e_n }{m-n}.
\end{equation}
In numerical experiments, it is obtained by averaging over $100$ independent experiments under different problem dimensions $d_u$ and ensemble sizes $J$. For the repeated experiments here, the problem setup (the transport speed $a$), data, initial ensemble and dropout are all randomized. \cref{fig:LTEConvRate} shows the linear dependence of the convergence rate on $d_u^{-1}$ and $J$. This verifies the theoretical prediction in \cref{thm:convergence}, which is expected since \Cref{asm:main} is strictly satisfied in this model.

\begin{figure}[htbp]
    \centering
    \includegraphics[scale = 0.33, trim = 50 200 50 200]{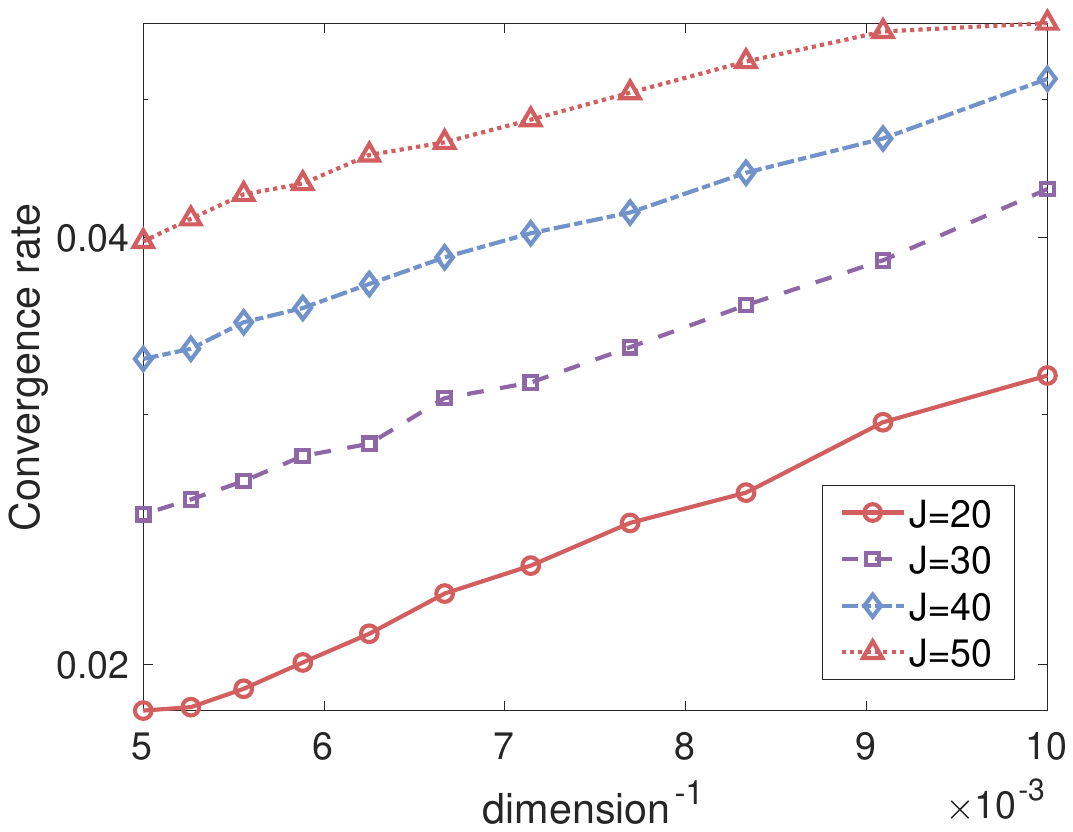}
    \includegraphics[scale = 0.33, trim = 0 200 50 200]{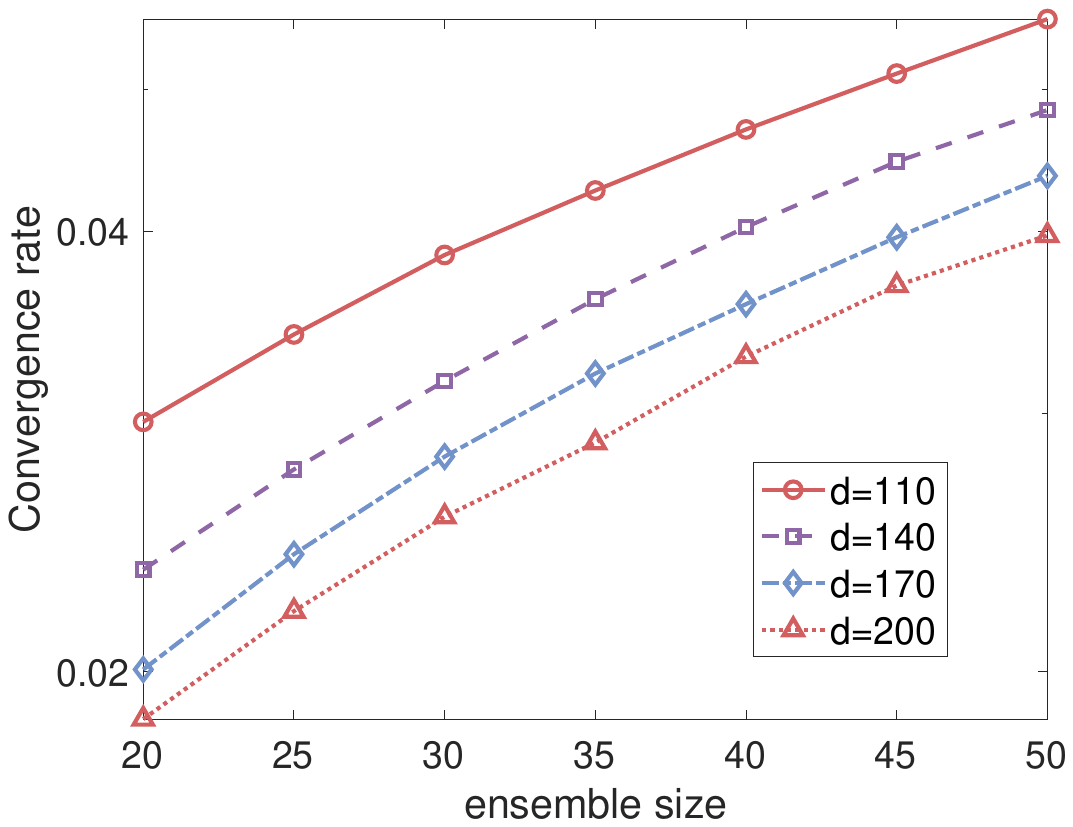}
    \caption{Linear dependence of the convergence rate \cref{eq:ConvRate} of DEKI on reciprocal of dimension $d_u^{-1}$ and ensemble size $J$, tested in linear transport equation model. }
    \label{fig:LTEConvRate}
\end{figure}

Finally, we comment that the CPU time of the EKI type method is mainly determined by the number of forward evaluations, since the it is usually the most time-consuming part. In each iteration, DEKI needs twice forward evaluations compared to
EKI. So for fixed iteration numbers, DEKI would costs roughly double CPU time.
But if we fix a certain accuracy, DEKI can still win as it converges faster.

\subsection{Darcy's Law}
Consider the 2d elliptic equation on $\Omega = [0,1]^2$:
\begin{equation}
\label{eq:2DEllip}
    -\nabla \cdot \Brac{ \exp (a(\bx)) \nabla v(\bx) } = f(\bx), \quad v|_{\partial \Omega} = 0.
\end{equation}
Consider estimating the log-permeability field $a(\bx)$ from observations on the field $v(\bx)$: pick evaluation points $\bx_1,\dots,\bx_{d_y}$, and 
\[
    y = (y_1,\dots,y_{d_y}), \quad y_j = v(\bx_j) + \xi_j, 
\]
where $v(\bx)$ is the solution of \cref{eq:2DEllip} given parameter $a(\bx)$ and $\xi_j\sim\mcN(0,\sigma_0^2) ~(\sigma_0 = 10^{-3})$ is the observational noise. Thus the forward map is
\begin{equation}
    \mcG: a(\bx) \mapsto y.
\end{equation}
The testing parameter $a_{\rm true}(\bx)$ is generated from a Gaussian random field with mean $\mean{a} \in \mR$ and covariance operator
\[
    K(\bx_1;\bx_2) = K(x_1,y_1;x_2,y_2) := \sigma^2 \exp \Brac{ - \frac{(x_1-x_2)^2}{2l_x^2} - \frac{(y_1-y_2)^2}{2l_y^2} }. 
\]
In computation, we represent $a(\bx)$ in the truncated Karhunen–Lo\'eve basis: 
\[
    a(\bx) \approx \mean{a} + \sum_{i=1}^{d_u} a_i \varphi_i(\bx), \quad \varphi_i(\bx) = \sqrt{\lambda_i}  \psi_i(\bx).
\]
where $(\lambda_i,\psi_i(\bx))$ are the eigenpairs of the covariance operator $K$ arranged in descending order $\lambda_1\geq \lambda_2 \geq \cdots$. Denote the coordinate representation $\hat{a} = (a_1,\dots,a_{d_u}) \in \mR^{d_u}$. The truncated dimension $d_u$ is determined by a threshold $\epsilon$:
\[
    d_u(\epsilon) = \inf\left\{ d \in\mZ : \sum_{i>d} \lambda_i \leq \epsilon \cdot \tr(K) \right \}.
\]
Namely, $d_u$ is the smallest number of basis functions that can approximate $a(\bx)$ with $\epsilon$-accuracy. The regularization term under this basis can be simply chosen as
\[
    R(\hat{a}) = \gamma^2 \norm{\hat{a}}_{l^2}^2 = \gamma^2 \sum_{i=1}^{d_u} a_i^2.
\]
Or equivalently, the regularization operator $\mcC_0^{-1/2} = \gamma I$.

We test under different setups listed in \cref{tab:SetupPara}, and fix external force $\displaystyle f(x,y) = 13\pi^2 \sin(2\pi x) \sin (3\pi y) $. The forward problem is solved using finite element method on a $32\times 32$ grid. The observation points are chosen as the center of the $8\times 8$ subblocks of $\Omega$, so that the data dimension $d_y=64$.

\begin{table}[htbp]
\centering
\renewcommand{\arraystretch}{1.25}
\caption{Parameters for different setups in Darcy's law model}
\label{tab:SetupPara}
\begin{tabular}{ccccccc}
\hline
        & $m$ & $\sigma$ & $l_x$ & $l_y$ & $\varepsilon$ & $d_u$ \\ \hline
setup 1 & $0$ &   $0.1$   &  $0.1$  & $0.1$  &  $10^{-3}$ & $136$  \\ \hline
setup 2 & $0$ &   $0.1$   &  $0.2$ &  $0.05$ &  $10^{-3}$ &  $142$ \\ \hline
setup 3 & $0$ &   $0.1$   &  $0.15$  & $0.05$  &  $10^{-3}$ & $179$ \\ \hline
setup 4 & $0$ &   $0.1$   &  $0.1$  & $0.05$ &  $10^{-3}$ & $254$
\\ \hline
\end{tabular}
\end{table}

We will compare DEKI with EKI. We do not use LEKI here, since it is not clear how to design localization under the KL basis. The ensemble is initialized as Gaussian variables $\hat{a}_0^{(j)} \overset{\iid} \sim \mcN(0,\gamma^2 I_{d_u})$, and choose the regularization parameter $\gamma = 0.1$. Fix reference step size $\tilde{h}=0.5$ and set the ratio $h_n/\tilde{h}_n=0.1$, see \cref{eq:StepSizes}. 

In each experiment, we run $N=10^4$ steps to show the asymptotic behavior. Note in practice, such long steps are not necessary. In \cref{fig:EllipSetup1cmp}, we compare the decay of the relative data misfit \cref{eq:RelaMisfit} and the relative solution error of DEKI and EKI in setup 1. Here the relative solution error is defined by 
\begin{equation}
\label{eq:RelaError}
    \text{err}(n) = \frac{\norm{a_n(\bx)-a_{\rm true}(\bx)}_{L^2}}{\norm{a_{\rm true}(\bx)}_{L^2}}.
\end{equation}
We repeat each experiment $50$ times under the same settings and plot the mean and standard deviation of these error. Here different realizations of initial ensemble and dropout are used for the repeated experiments.

Again we see that the EKI gets stuck after some steps but DEKI works well. Here the data misfit of DEKI does not decay exponentially. This is due to that the Polyak-Lojasiewic condition in \Cref{asm:main} is not satisfied globally for the Darcy's law model. Nevertheless, the DEKI still performs well both as an optimizer and as an inversion method in the high dimensional settings. 

\begin{figure}[htbp]
    \centering
    \includegraphics[scale = 0.33, trim = 50 200 50 200]{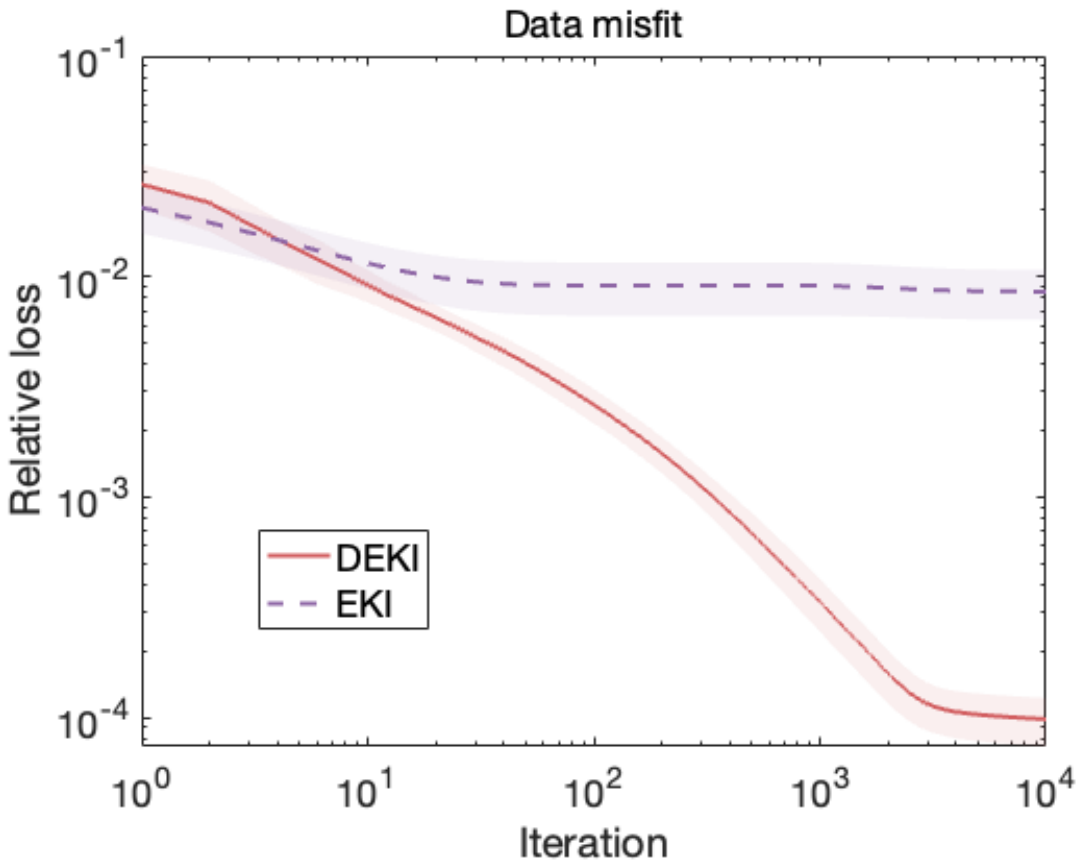}
    \includegraphics[scale = 0.33, trim = 0 200 50 200]{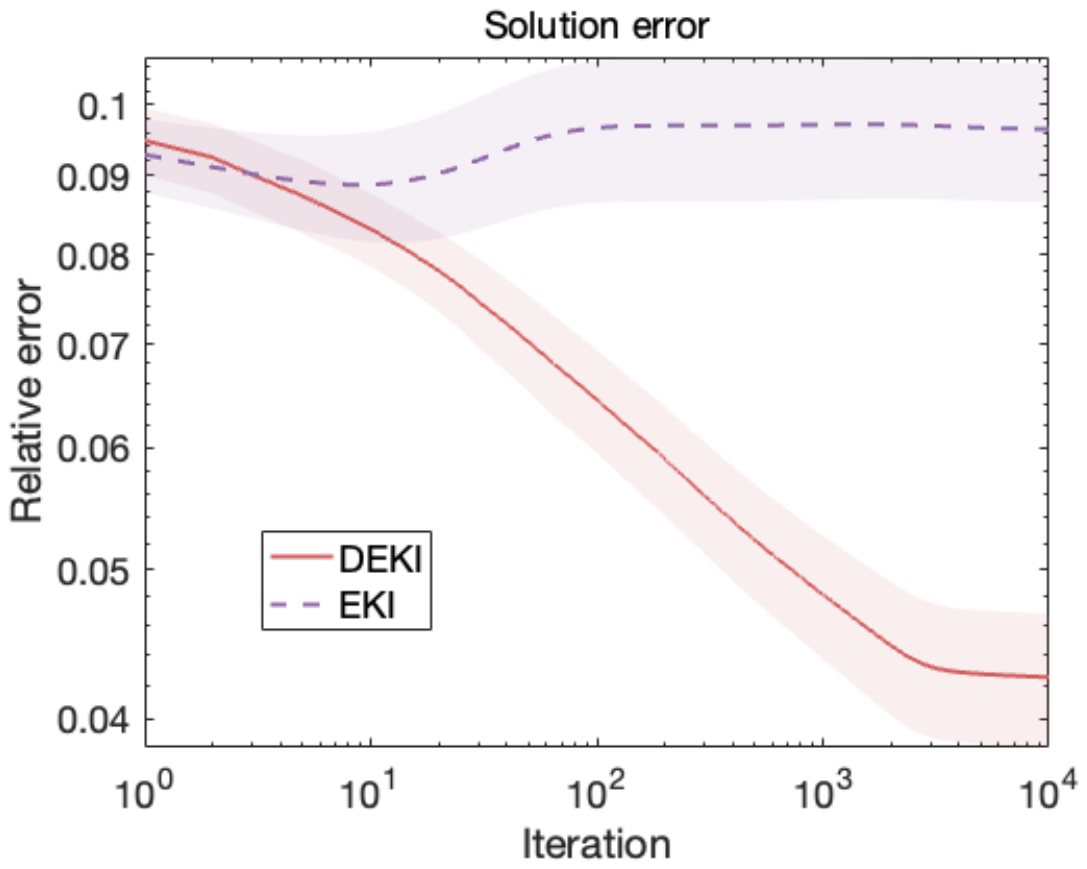}
    \caption{Comparison of DEKI and EKI in Darcy's law model in setup 1 ($d_u=136$) using ensemble size $J=15$. Left: relative data misfit \cref{eq:RelaMisfit}; right: relative solution error \cref{eq:RelaError}.}
    \label{fig:EllipSetup1cmp}
\end{figure}

Plots of the inversion solutions are displayed in \cref{fig:EllipSolu}. We can see that as dimension grows, DEKI is still able to infer the permeability field with good accuracy using small ensemble, while EKI solutions become incorrect.

\begin{figure}[htbp]
    \centering
    \includegraphics[scale = 0.6, trim = 0 140 0 100, clip]{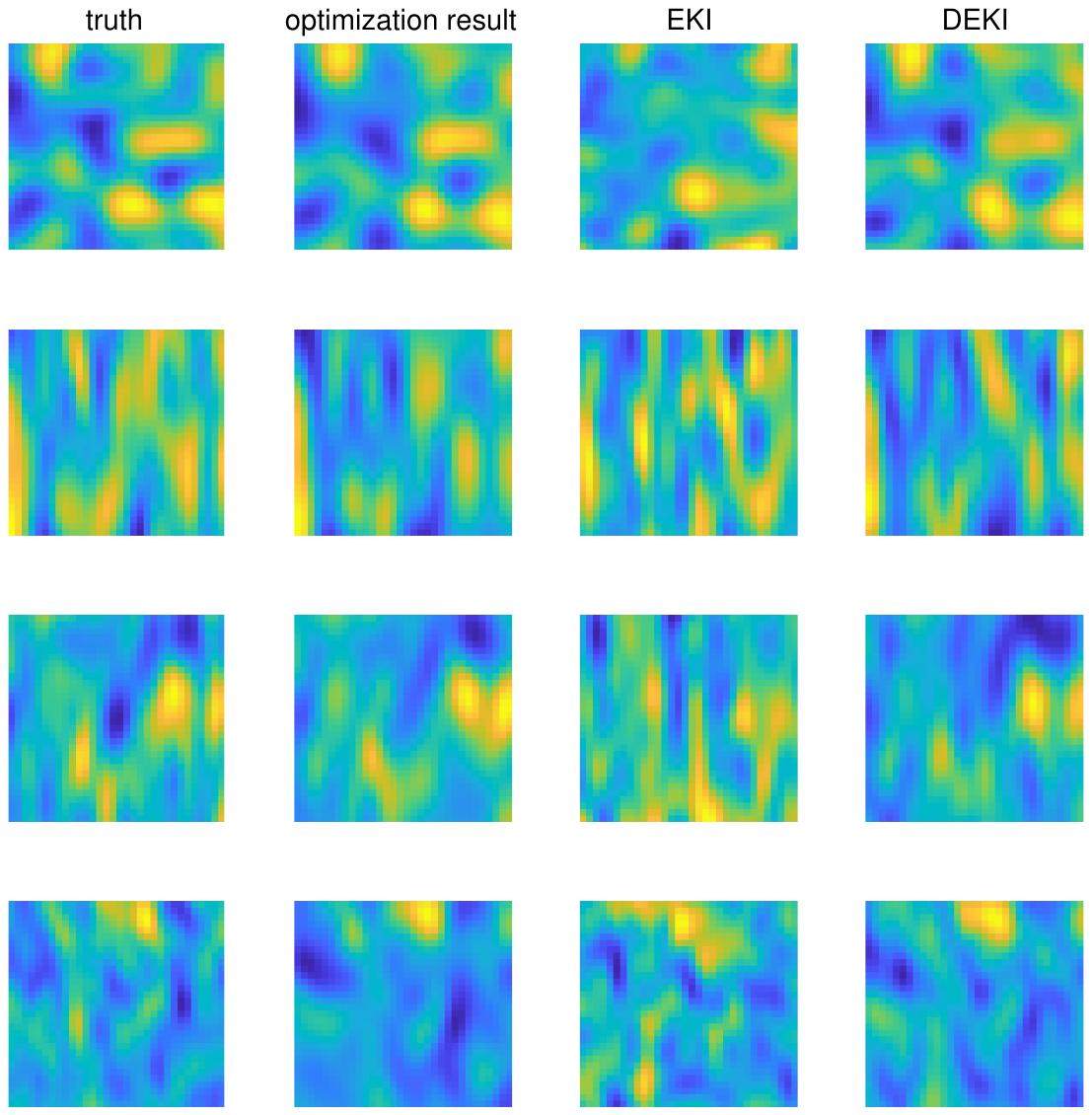}
    \caption{Comparison of the inversion solutions of EKI and DEKI in Darcy's law model in different setups with ensemble size $J=15$. Row 1: setup 1 $(l_x = 0.1, l_y = 0.1, d_u=136)$; Row 2: setup 2 $(l_x = 0.2, l_y = 0.05, d_u=142)$; Row 3: setup 3 $(l_x = 0.15, l_y = 0.05, d_u=179)$; Row 4: setup 4 $(l_x = 0.1, l_y = 0.05, d_u=254)$.}
    \label{fig:EllipSolu}
\end{figure}

\section{Conclusion}
We investigate the application of the dropout technique in EKI method and propose the DEKI scheme. We show that this method mitigates the subspace issue of the vanilla EKI, and performs well in high dimensional inverse problems using small ensemble, both theoretically and numerically. With careful analysis on the dimension dependence, we prove that DEKI converges exponentially for strongly convex problems and the computational complexity scales linearly with dimension. We also show that such scaling is optimal by conducting the query complexity analysis adapting from the optimization literature, which is new for inverse problems.

The idea of incorporating dropout technique into EKI follows the trial of EnKBF with dropout in \cite{MR4568201}. The numerical analysis here is a first step to understand the effectiveness of dropout in these methods. It would be interesting to investigate if the idea can be extended and work as a simpler but effective technique for other ensemble methods in high dimensional problems. It is also interesting to look at the interplay of our DEKI method with traditional zeroth order optimization methods, providing new perspective to understand them and generating new algorithms.

\section*{Acknowledgments}
The work of SL has been 
partially funded by Singapore MOE grant A-8000459-00-00.
The work of SR has been partially funded by Deutsche Forschungsgemeinschaft (DFG) - Project-ID 318763901 - SFB1294.
SR would also like to thank the Isaac Newton Institute for Mathematical Sciences, Cambridge, for support and hospitality during the programme  {\it The Mathematical and Statistical Foundation of Future Data-Driven Engineering} where work on this paper was undertaken. This work was supported by EPSRC grant no EP/R014604/1. The work of XT has been  funded by Singapore MOE grants A-0004263-00-00 and A-8000459-00-00.

\bibliographystyle{siamplain}
\bibliography{EKIdropout}


\appendix
\section{Missing proofs}
\subsection{Proof for Proposition \ref{prop:LinearErr}}
\label{app:linerr}
Denote for simplicity that
\begin{equation*}
    \begin{split}
        &T_n = \frac{1}{\sqrt{J-1}} [ \tilde{\tau}_n^{(1)}, \dots, \tilde{\tau}_n^{(J)} ], \quad \hat{Y}_n = G_n T_n, \\
        &Y_n = \frac{1}{\sqrt{J-1}} [ \tilde{\mfr}_n^{(1)}, \dots, \tilde{\mfr}_n^{(J)} ], \quad \tilde{\mfr}_n^{(j)} = \mcH(\tilde{u}_n^{(j)}) - \mean{\mcH(\tilde{u}_n)}.
    \end{split}
\end{equation*}
Then it holds $\tilde{C}_n^{uu} = T_n T_n\matT , \tilde{C}_n^{uz} = T_n Y_n\matT , \tilde{C}_n^{zz} = Y_n Y_n\matT $, and we can write \cref{eq:DEKImean} as
\begin{equation}
\label{eq:pfumeanup}
    \mean{u}_{n+1} = \mean{u}_n + \tilde{h}_n T_n Y_n\matT  ( I + \tilde{h}_n Y_n Y_n\matT  )^{-1} ( z - \mcH(\mean{u}_{n}) ),
\end{equation}
and similarly for \cref{eq:pfGN}:
\begin{equation}
\label{eq:pfupmeanup}
    \mean{u}_{n+1}' = \mean{u}_n + \tilde{h}_n T_n \hat{Y}_n\matT  ( I + \tilde{h}_n \hat{Y}_n\matT  \hat{Y}_n )^{-1} ( z - \mcH(\mean{u}_{n}) ).
\end{equation}
Denote $ Z_t = (1-t) Y_n + t \hat{Y}_n $, and define
\begin{equation*}
    S(t) = Z_t\matT  (I + \tilde{h}_n Z_t Z_t\matT  )^{-1},
\end{equation*}
Take difference of \cref{eq:pfumeanup} and \cref{eq:pfupmeanup},
\begin{equation}
\label{eq:pfnonlineardiff}
\begin{split}
    \normo{\mean{u}_{n+1} - \mean{u}_{n+1}' } =~& \normo{ \tilde{h}_n T_n (S(1) - S(0)) (z - \mcH(\mean{u}_{n}) } \\
    \leq~& \tilde{h}_n \normo{T_n} \normo{ S(1) - S(0) } \normo{ z - \mcH(\mean{u}_{n}) }.
\end{split}
\end{equation}
By direct computation and the Sherman-Morrison-Woodbury formula (\cref{lem:SMW}),
\begin{equation*}
    \begin{split}
        S'(t) =~& (Z_t\matT )' (I + \tilde{h}_n Z_t Z_t\matT  )^{-1} \\
        &- \tilde{h}_n Z_t\matT  (I + \tilde{h}_n Z_t Z_t\matT  )^{-1} [ Z_t' Z_t\matT  + Z_t (Z_t\matT )' ] (I + \tilde{h}_n Z_t Z_t\matT  )^{-1} \\
        =~& \Rectbrac{I - \tilde{h}_n Z_t\matT  (I + \tilde{h}_n Z_t Z_t\matT  )^{-1} Z_t } (Z_t\matT )' (I + \tilde{h}_n Z_t Z_t\matT  )^{-1} \\
        &- \tilde{h}_n Z_t\matT  (I + \tilde{h}_n Z_t Z_t\matT  )^{-1} Z_t' Z_t\matT  (I + \tilde{h}_n Z_t Z_t\matT  )^{-1} \\
        =~&  (I + \tilde{h}_n Z_t\matT  Z_t )^{-1} (Z_t\matT )' (I + \tilde{h}_n Z_t Z_t\matT  )^{-1} \\
        &- \tilde{h}_n Z_t\matT  (I + \tilde{h}_n Z_t Z_t\matT  )^{-1} Z_t' Z_t\matT  (I + \tilde{h}_n Z_t Z_t\matT  )^{-1}.
    \end{split}
\end{equation*}
Notice since $Z_tZ_t\matT $ is PSD, apply \cref{lem:PSDcmp} and we get
\begin{equation*}
    \begin{split}
        & \normo{ Z_t\matT  (I + \tilde{h}_n Z_t Z_t\matT  )^{-1}} \\
        =~& \normo{(I + \tilde{h}_n Z_t Z_t\matT  )^{-1} Z_t Z_t\matT  (I + \tilde{h}_n Z_t Z_t\matT  )^{-1}}^{1/2} \\
        \leq~& \normo{(I + \tilde{h}_n Z_t Z_t\matT  )^{-1}}^{1/2} \tilde{h}_n^{-1/2} \normo{ \tilde{h}_n Z_t Z_t\matT  (I + \tilde{h}_n Z_t Z_t\matT  )^{-1}}^{1/2} \leq \tilde{h}_n^{-1/2}.
    \end{split}
\end{equation*}
Therefore,
\begin{equation*}
\begin{split}
    \norm{S'(t)} \leq~& \normo{(I + \tilde{h}_n Z_t\matT  Z_t )^{-1}} \norm{(Z_t\matT )'} \normo{(I + \tilde{h}_n Z_t Z_t\matT  )^{-1}} \\
    &+ \tilde{h}_n \normo{Z_t\matT  (I + \tilde{h}_n Z_t Z_t\matT  )^{-1}}  \normo{ Z_t'} \normo{ Z_t\matT  (I + \tilde{h}_n Z_t Z_t\matT  )^{-1}} \\
    \leq~& \normo{Z_t'} + \tilde{h}_n (\tilde{h}_n^{-1/2})^2 \norm{Z_t'} = 2 \normo{Y_n-\hat{Y}_n}.
\end{split}
\end{equation*}
By definition,
\begin{equation}
\label{eq:pfYdiff}
    \begin{split}
        Y_n- \hat{Y}_n =~& \frac{1}{\sqrt{J-1}} [ \tilde{\mfr}_n^{(1)} - G_n \tilde{\tau}_n^{(1)} , \dots, \tilde{\mfr}_n^{(J)} - G_n \tilde{\tau}_n^{(J)} ] \\
        =~& \frac{1}{\sqrt{J-1}} [ r_n^{(1)} , \dots, r_n^{(J)} ] \Brac{I - \frac{1}{J}\ones\ones\matT },
    \end{split}
\end{equation}
where $\ones = (1,1,\dots,1)\matT $ and
\[
    r_n^{(j)} = \mcH( \tilde{u}_n^{(j)} ) - \mcH( \mean{u}_n ) - G_n \tilde{\tau}_n^{(j)} = \begin{bmatrix}
         \mcG( \tilde{u}_n^{(j)} ) - \mcG( \mean{u}_n ) -  \nabla \mcG( \mean{u}_n ) \tilde{\tau}_n^{(j)} \\
          0
    \end{bmatrix}.
\]
Using the bounds $\norm{\nabla^2\mcG} \leq H$ in \cref{eq:HessBound}, we obtain that
\[
    \normo{r_n^{(j)}} \leq \frac{1}{2} H \normo{\tilde{\tau}_n^{(j)}}^2.
\]
Notice $I - \frac{1}{J}\ones\ones\matT $ is an orthogonal projector and thus from \cref{eq:pfYdiff},
\begin{equation*}
    \begin{split}
        \normo{Y_n- \hat{Y}_n}^2 \leq~& \frac{1}{J-1} \normo{[ r_n^{(1)} , \dots, r_n^{(J)} ] }^2 \\
        \leq~& \frac{1}{J-1} \sum_{j=1}^J \normo{r_n^{(j)}}^2 \leq \frac{H^2}{4(J-1)} \sum_{j=1}^J \normo{\tilde{\tau}_n^{(j)}}^4.
    \end{split}
\end{equation*}
Therefore,
\begin{equation*}
\label{eq:pfYBound}
    \begin{split}
        \normo{Y_n- \hat{Y}_n} \leq~& \frac{H}{2\sqrt{J-1}} \Brac{\sum_{j=1}^J \normo{\tilde{\tau}_n^{(j)}}^4}^{1/2} \\
        \leq~& \frac{H}{2\sqrt{J-1}} \sum_{j=1}^J \normo{\tilde{\tau}_n^{(j)}}^2 \\
        =~& \frac{H}{2} \sqrt{J-1} \tr(\tilde{C}_n^{uu}) \leq \frac{1}{2} (J-1)^{3/2} H \normo{\tilde{C}_n^{uu}}.
    \end{split}
\end{equation*}
The last inequality holds since $\rk(\tilde{C}_n^{uu}) \leq J-1$. Combining the above results into \cref{eq:pfnonlineardiff} and noticing $\normo{\tilde{C}_n^{uu}} = \normo{T_nT_n\matT } = \normo{T_n}^2$, we obtain
\begin{equation*}
    \begin{split}
        \normo{\mean{u}_{n+1} - \mean{u}_{n+1}' } \leq~& \tilde{h}_n \normo{T_n} \normo{ z - \mcH(\mean{u}_{n}) } \int_0^1 \normo{ S'(t) } \mdd t  \\
        \leq~& \tilde{h}_n \normo{\tilde{C}_n^{uu}}^{1/2} \normo{ z - \mcH(\mean{u}_{n}) } \cdot (J-1)^{3/2} H \normo{\tilde{C}_n^{uu}}.
    \end{split}
\end{equation*}
Thus the conclusion holds for $C = (J-1)^{3/2} H $.

\subsection{Proof for Lemma \ref{lem:onestep}}
\label{app:onestep}
Since $l$ is $L$-smooth, 
\begin{equation}
\label{eq:pfConv}
    \begin{split}
        l(\mean{u}_{n+1}) \leq~& l(\mean{u}_n) + \ip{\nabla l(\mean{u}_n)}{\mean{u}_{n+1}-\mean{u}_{n}} + \frac{1}{2}L \norm{\mean{u}_{n+1}-\mean{u}_{n}}^2 \\
        =~& l(\mean{u}_n) + \ip{\nabla l(\mean{u}_n)}{\mean{u}'_{n+1}-\mean{u}_{n}} + \frac{1}{2}L \norm{\mean{u}_{n+1}'-\mean{u}_{n}}^2  \\
        +& \ip{\nabla l(\mean{u}_n) + L ( \mean{u}_{n+1}' - \mean{u}_{n}) }{\mean{u}_{n+1}-\mean{u}'_{n+1}} + \frac{1}{2}L \normo{\mean{u}_{n+1}-\mean{u}'_{n+1}}^2.
    \end{split}
\end{equation}
where $\mean{u}_{n+1}'$ is defined in \cref{eq:pfGN}. The second term provides contraction, and the rest terms are the residues to be controlled. Note though the third term is a higher order term $(\mcO(\tilde{h}_n^2))$, it is not small due to the adaptive step size. Thus we will combine the second and third terms together to derive sufficient decay. And the last two terms is the nonlinear residue, which will be controlled using \cref{prop:LinearErr}.

By definition, we can write the second term in \cref{eq:pfConv} as
\begin{equation}
\label{eq:pf2ndterm}
\begin{split}
    & \ip{\nabla l(\mean{u}_n)}{\mean{u}_{n+1}' - \mean{u}_n} \\
    =& - \ip{G_n\matT  ( z - \mcH(\mean{u}_{n}) ) }{\tilde{h}_n \tilde{C}_{n}^{uu} G_n\matT  ( I + \tilde{h}_n G_n \tilde{C}_{n}^{uu} G_n\matT  )^{-1} ( z - \mcH(\mean{u}_{n}) )} \\
    =& -\ip{z - \mcH(\mean{u}_{n})}{\psi_1(\tilde{h}_n G_n \tilde{C}_{n}^{uu} G_n\matT )(z - \mcH(\mean{u}_{n}))},
\end{split}
\end{equation}
where $\psi_1(x) = x(1+x)^{-1}$. For the third term in \cref{eq:pfConv}, 
\begin{equation}
\label{eq:pf3rdterm}
    \begin{split}
        &\frac{1}{2}L \norm{\mean{u}_{n+1}'-\mean{u}_{n}}^2 \\
        =~& \frac{1}{2} L \normo{ \tilde{h}_n \tilde{C}_{n}^{uu} G_n\matT  ( I + \tilde{h}_n G_n \tilde{C}_{n}^{uu} G_n\matT  )^{-1} ( z - \mcH(\mean{u}_{n}) )}^2 \\
        \leq~& \frac{1}{2} L \tilde{h}_n \normo{\tilde{C}_n^{uu}} \normo{(\tilde{h}_n \tilde{C}_{n}^{uu})^{1/2} G_n\matT  ( I + \tilde{h}_n G_n \tilde{C}_{n}^{uu} G_n\matT  )^{-1} ( z - \mcH(\mean{u}_{n}) )}^2 \\
        =~& \frac{1}{2} L \tilde{h}_n \normo{\tilde{C}_n^{uu}} \ip{ z - \mcH(\mean{u}_{n}) }{\psi_2(\tilde{h}_n G_n \tilde{C}_{n}^{uu} G_n\matT  ) ( z - \mcH(\mean{u}_{n}) )},
    \end{split}
\end{equation}
where $\psi_2(x) = x(1+x)^{-2}$. Notice $\tilde{C}_n^{uu}$ is a submatrix of $C_n^{uu} \St \normo{\tilde{C}_n^{uu}} \leq \normo{C_n^{uu}}$. By assumption $\mu\leq L^{-1}$, and thus the choice $\tilde{h}_n = \mu \normo{C_n^{uu}}^{-1}$ implies that
\[
\frac{1}{2} L \tilde{h}_n \normo{\tilde{C}_n^{uu}} \leq \frac{1}{2} L \mu \leq \frac{1}{2}.
\]
Therefore, combine \cref{eq:pf2ndterm} and \cref{eq:pf3rdterm} and we obtain
\begin{equation}
\label{eq:pfdecaytermbound}
    \begin{split}
        & \ip{\nabla l(\mean{u}_n)}{\mean{u}_{n+1}' - \mean{u}_n} + \frac{1}{2}L \norm{\mean{u}_{n+1}'-\mean{u}_{n}}^2  \\
        \leq& - \ip{z - \mcH(\mean{u}_{n})}{ \psi(\tilde{h}_n G_n \tilde{C}_{n}^{uu} G_n\matT )( z - \mcH(\mean{u}_{n}) )},
    \end{split}
\end{equation}
where 
\[
    \psi(x) = \psi_1(x) - \frac{1}{2} \psi_2(x) = \frac{x}{1+x} - \frac{x}{2(1+x)^2} = \frac{ x(1+ 2x)}{2(1+ x)^2},
\]
Here we use the fact that the positive definite matrix $ B := \tilde{h}_n G_n \tilde{C}_{n}^{uu} G_n\matT $ is diagonalizable: $ B = V_B \Lambda_B V_B\matT$, so that polynomial functions of $B$ can be computed as if $B$ is a scalar, i.e. 
\[
    B(B+I)^{-1} - \frac{1}{2} B (B+I)^{-2} = V_B \Rectbrac{ \Lambda_B (\Lambda_B+I)^{-1} - \frac{1}{2} \Lambda_B (\Lambda_B+I)^{-2} } V_B\matT. 
\]
Notice the upper bound holds
\[
    \normo{\tilde{h}_n G_n \tilde{C}_{n}^{uu} G_n\matT } \leq \tilde{h}_n M^2 \normo{\tilde{C}_{n}^{uu}} \leq \mu M^2.
\]
It is elementary to show that $\psi(x)/x$ is monotone decreasing, and thus when $0< x \leq \mu M^2$, it holds $\psi (x) \geq \mu^{-1} M^{-2} \psi(\mu M^2) x $. Therefore, 
\begin{align}
    &\mE_n \psi(\tilde{h}_n G_n \tilde{C}_{n}^{uu} G_n\matT ) \notag \\ 
    \succeq~& \mu^{-1} M^{-2} \psi(\mu M^2) \cdot \mE_n \tilde{h}_n  G_n \tilde{C}_{n}^{uu} G_n\matT  &\text{(\cref{lem:PSDcmp})} \notag\\
    \succeq~& \mu^{-1} M^{-2} \psi(\mu M^2) \cdot \lambda(1-\lambda) \tilde{h}_n  \min_s C_n^{uu}(s,s) G_n G_n \matT  &\text{by \cref{eq:ExpC}} \notag \\
    \succeq~& \mu^{-1} M^{-2} \psi(\mu M^2) \cdot \lambda(1-\lambda) \tilde{h}_n \bar{\kappa}^{-1} \normo{C_n^{uu}} \min_s P(s,s) G_n G_n \matT  &\text{(\cref{prop:ensemble_collapse_diag})} \notag \\
    =~& M^{-2} \psi(\mu M^2) \lambda(1-\lambda) \bar{\kappa}^{-1} \min_s P(s,s) G_n G_n \matT  \notag.
\end{align}
Plug it in \cref{eq:pfdecaytermbound} and use the PL condition $\normo{\nabla l(u)}^2 \geq c(l(u)-l_{\min})$, we obtain
\begin{equation}
\label{eq:pfdecay}
\begin{split}
    & \mE_n \Rectbrac{ \ip{\nabla l(\mean{u}_n)}{\mean{u}_{n+1}' - \mean{u}_n} + \frac{1}{2}L \norm{\mean{u}_{n+1}'-\mean{u}_{n}}^2 }  \\
    \leq& - \mE_n \ip{z - \mcH(\mean{u}_{n})}{ \psi(\tilde{h}_n G_n \tilde{C}_{n}^{uu} G_n\matT )( z - \mcH(\mean{u}_{n}) )}  \\
    \leq& -M^{-2} \psi(\mu M^2) \lambda(1-\lambda) \bar{\kappa}^{-1} \min_s P(s,s) \normo{G_n\matT (z - \mcH(\mean{u}_{n}))}^2 \\
    \leq& - 2 \beta_0 (l(\bar{u}_n) - l_{\min}),
\end{split}
\end{equation}
where we denote $\beta_0 = \frac{1}{2} c M^{-2} \psi(\mu M^2) \lambda(1-\lambda) \bar{\kappa}^{-1} \min_s P(s,s)$.

Next we control the nonlinear residues. For the fourth term in \cref{eq:pfConv}, first notice
\begin{equation}
\label{eq:pfincrebound}
    \begin{split}
        \normo{\mean{u}_{n+1}' - \mean{u}_{n} } \leq~& \normo{\tilde{h}_n \tilde{C}_n^{uu} G_n\matT  (I+\tilde{h}_n G_n \tilde{C}_n^{uu}G_n\matT )^{-1}} \normo{z - \mcH(\mean{u}_{n})} \\
        \leq~& \mu^{1/2} \normo{(\tilde{h}_n \tilde{C}_n^{uu})^{1/2} G_n\matT  (I+\tilde{h}_n G_n \tilde{C}_n^{uu}G_n\matT )^{-1}} \normo{z - \mcH(\mean{u}_{n})} \\
        \leq~& \mu^{1/2} \normo{z - \mcH(\mean{u}_{n})},
    \end{split}
\end{equation}
where we uses the following inequality: since $G_n \tilde{C}_n^{uu}G_n\matT \succeq 0$,
\begin{equation*}
\begin{split}
    & \normo{(\tilde{h}_n \tilde{C}_n^{uu})^{1/2} G_n\matT  (I+\tilde{h}_n G_n \tilde{C}_n^{uu}G_n\matT )^{-1}}^2 \\
    =~&  \normo{ (I+\tilde{h}_n G_n \tilde{C}_n^{uu}G_n\matT )^{-1} \tilde{h}_n  G_n \tilde{C}_n^{uu} G_n\matT  (I+\tilde{h}_n G_n \tilde{C}_n^{uu}G_n\matT )^{-1}} \\
    \leq~& \normo{ (I+\tilde{h}_n G_n \tilde{C}_n^{uu}G_n\matT )^{-1} } \normo{ \tilde{h}_n  G_n \tilde{C}_n^{uu} G_n\matT  (I+\tilde{h}_n G_n \tilde{C}_n^{uu}G_n\matT )^{-1}} \leq 1.
\end{split}
\end{equation*}
Using \cref{eq:pfincrebound}, \cref{eq:LinearErr} and $\tilde{h}_n \leq \mu \normo{\tilde{C}_n^{uu}}^{-1} $, we get 
\begin{equation*}
\label{eq:pfres1}
    \begin{split}
        &\ip{\nabla l(\mean{u}_n) + L ( \mean{u}_{n+1}' - \mean{u}_{n}) }{\mean{u}_{n+1}-\mean{u}'_{n+1}} \\
        \leq~& \Brac{\normo{G\matT _n(z - \mcH(\mean{u}_{n}))} + L \normo{ \mean{u}_{n+1}' - \mean{u}_{n} } } \normo{\mean{u}_{n+1}-\mean{u}'_{n+1}}  \\
        \leq~& C ( M + L \mu^{1/2} ) \tilde{h}_n \normo{\tilde{C}_n^{uu}}^{3/2} \normo{ z - \mcH(\mean{u}_{n}) }^2 \\
        \leq~& C ( M + L \mu^{1/2} ) \mu \normo{\tilde{C}_n^{uu}}^{1/2} \normo{ z - \mcH(\mean{u}_{n}) }^2.
    \end{split}
\end{equation*}
The last term can be controlled in a similar way:
\begin{equation*}
\label{eq:pfres2}
\begin{split}
    \frac{1}{2}L \normo{\mean{u}_{n+1}-\mean{u}'_{n+1}}^2 \leq~& \frac{1}{2} L C^2 \tilde{h}_n^2 \normo{\tilde{C}_n^{uu}}^{3} \normo{ z - \mcH(\mean{u}_{n}) }^2 \\
    \leq~& \frac{1}{2} L C^2 \mu^2 \normo{\tilde{C}_n^{uu}}  \normo{ z - \mcH(\mean{u}_{n}) }^2.
\end{split}
\end{equation*}
Summing up the above two inequalities, we obtain
\begin{equation}
\label{eq:pfres}
\begin{split}
    \ip{\nabla l(\mean{u}_n) + L ( \mean{u}_{n+1}' - \mean{u}_{n}) }{\mean{u}_{n+1}-\mean{u}'_{n+1}} +  \frac{1}{2}L \normo{\mean{u}_{n+1}-\mean{u}'_{n+1}}^2& \\
    \leq \frac{1}{2} \Delta_n \normo{ z - \mcH(\mean{u}_{n}) }^2 = \Delta_n l(\mean{u}_n) &,
\end{split}
\end{equation}
where $\Delta_n = 2C \mu ( M + L \mu^{1/2} ) \normo{C_n^{uu}}^{1/2} + L C^2 \mu^2 \normo{C_n^{uu}}$. Note here we use the relation $\normo{\tilde{C}_n^{uu}} \leq \normo{C_n^{uu}}$.

Combine \cref{eq:pfConv}, \cref{eq:pfdecay} and \cref{eq:pfres}, and we prove our claim.

\subsection{Proof for \cref{lem:scaling}}   \label{app:scaling}

When $\{u_0^{(j)}\}_{j=1}^J$ is initialized with i.i.d. Gaussian variables, in the regime $J\ll d_u$, it holds that $r=\rk(C_0^{uu}) \approx J-1$, so that 
\[
    \tr(C_0^{uu}) = \mcO(d_u) ~\St~ \normo{C_0^{uu}} = \mcO(d_u/J).
\]
Similarly, the projection $P$ onto the initial subspace (see \cref{prop:ensemble_collapse_diag}) satisfies
\[
    \tr(P) = r = \mcO(J) ~\St~ \min_s P(s,s) = \mcO(J/d_u).
\]
From the estimation above and the definition \cref{eq:beta0}, we have 
\[
    \beta_0 = c \lambda(1-\lambda) \bar{\kappa}^{-1} \min_s P(s,s) \cdot \frac{\mu(1+ 2\mu M^2)}{4(1+\mu M^2)^2} = \mcO(J/d_u),
\]
since $\min_s P(s,s) = \mcO(J/d_u)$, and by assumption, $\mean{\kappa} = \mcO(1)$ and
\[
    c = \mcO(M^2) ~\St~ \frac{c \mu(1+ 2\mu M^2)}{4(1+\mu M^2)^2} \sim \frac{\mu M^2 (1+ 2\mu M^2)}{4(1+\mu M^2)^2} = \mcO(1).
\]
Since we take the step size $\theta \sim M^{-2}$, so that $\gamma^2 \theta \sim \gamma^2/M^2 \sim 1$, which implies
\[
    \frac{ \gamma^2 \theta }{2 ( 1 + \gamma^2 \theta ) } \sim 1 ~\St~ \beta = \min\{ \beta_0, \frac{ \gamma^2 \theta }{2 ( 1 + \gamma^2 \theta ) } \} = \beta_0 = \mcO(J/d_u).
\]
As for $n_0$, note 
\[
    n_0 = \Big\lceil \frac{\log (\beta^{-1}C_1) }{\log \delta^{-1}} \Big\rceil = \Big\lceil \frac{\log \Rectbrac{ \beta^{-1} \Brac{ 2 C \mu (M + L \mu^{1/2} ) \normo{C_0^{uu}}^{1/2} + L C^2 \mu^2 \normo{C_0^{uu}} } } }{\log (1+ \gamma^2 \theta)}\Big\rceil. 
\] 
The numerator is $\log \text{poly} (d_u) = \mcO(\log d_u)$ as all these constants depend on $d_u$ at most polynomially. The denominator $\log(1+\gamma^2 \theta) \sim \gamma^2 \theta \sim 1$. It then follows that $n_0 = \mcO(\log d_u)$. 

\subsection{Proof for Proposition \ref{prop:complexity}}
\label{app:complexity}
We will prove the lower bound by construction. For simplicity, normalize $\norm{y}=1$. We start from some arbitrary $\norm{G'}\leq 1$, and consider the matrix class
\[
P = \{ \norm{G}\leq 2 : GU = G'U \}.
\]
Note $\forall G \in P, A(U,GU,y) \equiv A(U,G'U,y)$. We will construct $G_0,G_B\in P$ s.t.
\begin{equation}
\label{eq:pfOptSoluDist}
    \normo{u^*(G_0) - u^*(G_B)} \geq 0.2.    
\end{equation}
Then at least one $G_*\in\{G_0,G_B\}$ would satisfy $\|A(U,G_*U,y)-u^*(G_*)\|\geq 0.1$, since
\begin{equation*}
    \begin{split}
        & \|A(U,G_0U,y)-u^*(G_0)\| +  \|A(U,G_BU,y)-u^*(G_B)\| \\
        \geq~& \| \Brac{A(U,G_0U,y)-u^*(G_0)} - \Brac{ A(U,G_BU,y)-u^*(G_B)} \| \\
        =~& \normo{u^*(G_0) - u^*(G_B)} \geq 0.2.
    \end{split}
\end{equation*}
And the conclusion follows.

The proof is divided into two parts. First we express $P$ and $u^*(G)$ in an explicit form. Then we construct $G_0,G_B \in P$ that satisfy \cref{eq:pfOptSoluDist}. \\
\\
{\bf Explicit form of $P$ and $u^*(G)$.}
$U$ admits a singular value decomposition
\[
    U = \Sigma 
    \begin{pmatrix}
        \Lambda & 0 \\
        0 & 0
    \end{pmatrix} 
    V
    = \Sigma 
    \begin{pmatrix}
        I_r & 0 \\
        0 & 0
    \end{pmatrix}
    \begin{pmatrix}
        \Lambda & 0 \\
        0 & I
    \end{pmatrix}
    V,
\]
where $r$ is the rank of $U$, and notice $r\leq n \leq d_u/2$. For any $G\in P$, denote  
\[
G\Sigma = (A,B), ~G'\Sigma = (C,D), \quad A,C \in \mR^{d_y\times r}, ~ B,D \in \mR^{d_y\times (d_u-r)}.
\]
Denote $R = \begin{pmatrix} \Lambda & 0 \\ 0 & I \end{pmatrix} V$, and notice 
\begin{equation*}
    \begin{split}
        GU = G'U \ioi~& G\Sigma \Sigma\matT  U R^{-1} = G'\Sigma \Sigma\matT  U R^{-1} \\
        \ioi~& (A,B) 
        \begin{pmatrix}
            I_r & 0 \\
            0 & 0
        \end{pmatrix}
        = (C,D) 
        \begin{pmatrix}
            I_r & 0 \\
            0 & 0
        \end{pmatrix}
        \ioi A = C.
    \end{split}
\end{equation*}
So that $G = (C,B)\Sigma\matT $. Then we can write
\[
    P = \{ (C,B) \Sigma\matT  : B \in \mR^{d_y\times(d_u-r)}, ~\norm{(C,B)} \leq 2 \},
\]
By direct computation, for $G = (C,B) \Sigma\matT  \in P$, it holds
\begin{equation}
\label{eq:pfOptSoluForm}
    \begin{split}
        u^*(G) =~& \argmin_u \frac{1}{2} \norm{Gu-y}^2 + \frac{1}{2} \norm{u}^2 = (G\matT G+I)^{-1} G\matT  y \\
        =~& \Sigma \Rectbrac{(C,B)\matT (C,B)+I}^{-1} (C,B)\matT  y.
    \end{split}
\end{equation}
{\bf Construction of $G_0$ and $G_B$.}
Consider the two matrices 
\[
    G_0 = (C,0)\Sigma\matT , \quad G_B = (C,B)\Sigma\matT ,
\]
where $B$ is to be determined later. By \cref{eq:pfOptSoluForm},
\[
    u^*(G_0) = \Sigma z_0, \quad z_0 = 
    \begin{bmatrix}
        (C\matT  C +I)^{-1} C\matT  y \\
        0 
    \end{bmatrix}.
\]
\[
    u^*(G_B) = \Sigma z_C, \quad z_C = 
    \begin{pmatrix}
        C\matT  C + I & C\matT B \\
        B\matT  C & B\matT B + I
    \end{pmatrix}^{-1} 
    \begin{pmatrix}
        C\matT  y \\
        B\matT  y
    \end{pmatrix}.
\]
Notice 
\[
    \normo{u^*(G_0) - u^*(G_B)} = \norm{z_0-z_C} \geq \norm{z_{C,2}},
\]
where $z_{C,2}$ is the second block of $z_C$, which can be computed directly by
\[
    z_{C,2} = \Brac{I + B\matT  \Brac{I+C C\matT }^{-1} B}^{-1} B\matT  \Brac{I+C C\matT }^{-1} y .
\]
Denote $C C\matT  = Q\Lambda Q\matT $ as the eigenvalue decomposition, and denote $\tilde{B} = Q\matT B$, then
\[
    z_{C,2} = \Brac{I + \tilde{B}\matT  \Brac{I+\Lambda^2}^{-1} \tilde{B}}^{-1} \tilde{B}\matT  \Brac{I+\Lambda^2}^{-1} Q\matT y .
\]
Denote $r' = r\wedge d_y$, and note $\Lambda \in \mR^{d_y\times d_y}$ must be of the form
\[
    \Lambda = 
    \begin{pmatrix}
        \Lambda_{r'} & 0 \\
        0 & 0
    \end{pmatrix}, \quad \Lambda_{r'} \in \mR^{r'\times r'}.
\]
We will choose $\tilde{B} \in \mR^{d_y\times (d_u-r)} $ of the form
\[
    \tilde{B} = 
    \begin{pmatrix}
        \Lambda' & 0 \\
        E & 0
    \end{pmatrix},\quad \Lambda' \in \mR^{r'\times r'},~ E\in\mR^{(d_y-r')\times r'},
\]
where $\Lambda'$ is diagonal. Note the decomposition is possible since $d_u - r \geq r \geq r'$. Also notice when $r\geq d_y$, $\Lambda,\tilde{B}$ only have the upper left block. This does not affect our deduction, since we will set $E=0$ in this case and the computation is still correct.

Accordingly decompose $Q\matT y = (z_1,z_2)$ where $z_1 \in\mR^{r'}, z_2 \in\mR^{d_y-r'}$. Notice $1 = \normo{Q\matT y}^2 = \norm{z_1}^2 + \norm{z_2}^2$. By direct cputation, $z_{C,2} = (z', 0)$, where
\[
    z' = \Brac{ I + \Lambda'(I+\Lambda_{r'}^2)^{-1} \Lambda' + E\matT  E }^{-1} \Brac{ \Lambda'(I+\Lambda_{r'}^2)^{-1} z_1 + E\matT  z_2 }.
\]
When $\norm{z_2}^2\geq 1/2$, we can take $\Lambda' =0, E = \norm{z_2}^{-1} z_2 e_1\matT $ where $e_1=(1,0,\cdots,0)$. (note it is possible since $\norm{z_2}>0$ implies $r'<d_y$) and then
\[
    z' = \Brac{I+E\matT E}^{-1} E\matT  z_2 = \frac{1}{ 2 } \norm{z_2} e_1 \St \norm{z'} = \frac{1}{ 2 } \norm{z_2} \geq \frac{1}{2\sqrt{2}}.
\]
Otherwise, $\norm{z_1}^2 = 1 - \norm{z_2}^2 \geq 1/2$, we can take $E=0$, $\Lambda' = I$, and then
\[
    z' = \Brac{ I + ( I + \Lambda_{r'}^2 )^{-1} }^{-1} (I+\Lambda_{r'}^2)^{-1} z_1 = (2I + \Lambda_{r'}^2)^{-1} z_1.
\]
\[
    \St \norm{z'} \geq  \frac{\norm{z_1}}{ 2 + \norm{\Lambda_{r'}}^2 } \geq \frac{1}{3\sqrt{2}}, \quad \text{since } \norm{\Lambda_{r'}} = \normo{CC\matT } \leq \normo{G'}^2 \leq 1.
\]
So that we prove in any case,
\[
    \norm{u^*(G_0) - u^*(G_B)} \geq \norm{z_{C,2}} = \norm{z'} \geq \frac{1}{3\sqrt{2}} > \frac{1}{5}.
\]
The last thing to verify is that $G_B$ satisfies the norm constraint $\norm{G_B}\leq2$. Notice
\[
\norm{G_B}^2 = \norm{(C,B)}^2 = \norm{(C,B)(C,B)\matT } = \norm{CC\matT  + BB\matT  } \leq \norm{G'}^2 + \normo{\tilde{B}\tilde{B}\matT }.
\]
In the above cases, $\normo{\tilde{B}} = \norm{E} = \norm{z_2}^{-1} \norm{z_2} \norm{e_1} = 1$ or $\normo{\tilde{B}} = \norm{\Lambda'} = \norm{I} = 1$. So that $\norm{G_B}^2 \leq 2 < 4$. This completes our proof.

\subsection{Stability}
\begin{lemma}
\label{lem:meanstab}
DEKI with step size $\tilde{h}_n = \mu \normo{C_n^{uu}}^{-1}$ is stable in the sense that
\begin{equation}
    \normo{z-\mcH(\mean{u}_{n+1})} \leq ( 1 + M \sqrt{\mu/2} ) \normo{z-\mcH(\mean{u}_n)}.
\end{equation}
\end{lemma}
\begin{proof}
From \cref{eq:pfumeanup}, we obtain
\begin{equation*}
\begin{split}
    \normo{\mean{u}_{n+1} - \mean{u}_n} \leq~& \tilde{h}_n \normo{ T_n Y_n\matT  ( I + \tilde{h}_n Y_n Y_n\matT  )^{-1} ( z - \mcH(\mean{u}_{n}) )} \\
    \leq~& \tilde{h}_n \normo{ T_n Y_n\matT  ( I + \tilde{h}_n Y_n Y_n\matT  )^{-1}} \normo{ z - \mcH(\mean{u}_{n}) }.
\end{split}
\end{equation*}
Notice $\normo{T_n\matT  T_n} = \norm{T_nT_n\matT } = \normo{\tilde{C}_n^{uu}} $, and thus
\begin{equation*}
\begin{split}
    &\normo{ T_n Y_n\matT  ( I + \tilde{h}_n Y_n Y_n\matT  )^{-1}}^2 \\
    =~& \normo{ ( I + \tilde{h}_n Y_n Y_n\matT  )^{-1} Y_n T_n\matT  T_n Y_n\matT  ( I + \tilde{h}_n Y_n Y_n\matT  )^{-1}} \\
    \leq~&  \normo{ C_n^{uu} } \normo{ ( I + \tilde{h}_n Y_n Y_n\matT  )^{-1} Y_n Y_n\matT  ( I + \tilde{h}_n Y_n Y_n\matT  )^{-1}}.
\end{split}
\end{equation*}
Using \cref{lem:PSDcmp} (3) with $\psi(x) = x(1+x)^{-2} \leq 1/2$, we obtain
\[
    \tilde{h}_n ( I + \tilde{h}_n Y_n Y_n\matT  )^{-1} Y_n Y_n ( I + \tilde{h}_n Y_n Y_n\matT  )^{-1} \preceq I/2.
\]
Combining the above estimates, we obtain
\begin{equation*}
\begin{split}
    \normo{\mean{u}_{n+1} - \mean{u}_n} \leq~& \sqrt{\frac{\tilde{h}_n \normo{ C_n^{uu} }}{2}} \normo{ z - \mcH(\mean{u}_{n}) } = \sqrt{\frac{ \mu }{2}} \normo{ z - \mcH(\mean{u}_{n}) }.
\end{split}
\end{equation*}
Since $\mcH$ is $M$-Lipchitz,
\begin{equation*}
\begin{split}
    \normo{z-\mcH(\mean{u}_{n+1})} \leq~& \normo{z-\mcH(\mean{u}_{n})} + M \normo{\mean{u}_{n+1} - \mean{u}_n} \\
    \leq~& ( 1 + M \sqrt{\mu/2} ) \normo{z-\mcH(\mean{u}_n)}.
\end{split}
\end{equation*}
\end{proof}

\section{Some matrix analysis results}
Here we list some lemmas on the matrix analysis. More details can be found in textbooks, for instance \cite{MR2978290}.
\begin{lemma}
\label{lem:SMW}
    (Sherman-Morrison-Woodbury) Let $A\in \mR^{n\times n}$ be an invertible matrix. Let $U, V \in \mR^{n\times m}$. If $I+V\matT A^{-1}U$ is nonsingular, then
    \begin{equation}
        (A+UV\matT )^{-1} = A^{-1} - A^{-1} U ( I+V\matT A^{-1}U )^{-1} V\matT  A^{-1}.
    \end{equation}
\end{lemma}

\begin{lemma}
\label{lem:PSDcmp}
(Comparison) Let $A,B\in\mR^{n\times n}$ be two positive semidefinite matrices such that $A\preceq B$, then it holds that
\begin{enumerate}[(1)]
    \item $\lambda_k(A) \leq \lambda_k (B)$, where $\lambda_k(M)$ is the $k$-th largest eigenvalue of a matrix $M$.
    \item $\forall C \in \mR^{m\times n}, CAC\matT  \preceq CBC\matT  $.
    \item $\forall \psi, \varphi ~\st \psi(x) \leq \varphi(x), ~\psi(A) \preceq \varphi(A)$. As a corollary, when $\lambda_k(A) \in [a,b]$ for all $k$ and $\psi(x) \leq \varphi(x)$ when $x\in[a,b]$, then $\psi(A) \preceq \varphi(A)$.
    \item $\forall \psi: \mR \gto \mR$ that is monotone increasing, $\psi(A)\preceq \psi(B)$. Similarly, for monotone decreasing function $\psi$, $\psi(A)\succeq \psi(B)$.
\end{enumerate}
\end{lemma}
\begin{proof}
For (1), from Courant minimax principle,
\[
    \lambda_k(A) = \max_{V\subset \mR^n ,\dim V = k} \min_{0\neq u\in V} \frac{u\matT A u}{u\matT u}.
\]
Since $A \preceq B \St u\matT A u \leq u\matT B u$, we conclude that $\lambda_k(A)\leq \lambda_k(B)$. For (2), notice
\[
    \forall u \in\mR^m,~ u\matT  CAC\matT  u = (C\matT u)\matT  A (C\matT u) \leq (C\matT u)\matT  B (C\matT u) = u\matT  CBC\matT  u.
\]
For (3), denote the eigen decomposition of $A$ is $A = Q \Lambda Q\matT $, then by definition $\psi(A) = Q \psi(\Lambda) Q\matT $ and $\varphi(A) = Q \varphi(\Lambda) Q\matT $. Notice
\[
    \psi(\Lambda) = \diag \{\psi(\lambda_1),\dots, \psi(\lambda_n)\} \preceq \diag \{\varphi(\lambda_1),\dots, \varphi(\lambda_n)\} = \varphi(\Lambda),
\]
since $\psi(\lambda_k)\leq \varphi(\lambda_k)$. By (2), we conclude that $\psi(A) \preceq \varphi(A)$. For the corollary, just notice we can modify $\psi(x) = \varphi(x)$ outside $[a,b]$ and $\psi(A), \varphi(A)$ are unchanged. For (4), notice from (1), $\lambda_k(A)\leq \lambda_k(B)$, the conclusion follows by (2) and the monotonicity of $\psi$ using the same claim.
\end{proof}

\begin{lemma}
\label{lem:Ccmp}
Let $A,B \in \mR^{n\times n}$ be two positive semidefinite matrices. If there exists
 $0<m\leq M, \st m I \preceq A \preceq MI$, then it holds
\begin{equation}
    B (I+M B)^{-2} \preceq (I+BA)^{-1} B (I+AB)^{-1} \preceq B (I+m B)^{-2}.
\end{equation}
\end{lemma}
\begin{proof}
We first prove the inequality for nonsingular $B$. Notice
\begin{equation*}
\begin{split}
    &(I+BA)^{-1} B (I+AB)^{-1} \\
    =~& (B^{1/2}(I+B^{1/2}AB^{1/2})B^{-1/2})^{-1} B ( B^{-1/2}(I+B^{1/2}AB^{1/2})B^{1/2} )^{-1} \\
    =~& B^{1/2} ( I + B^{1/2} A B^{1/2})^{-2} B^{1/2}.
\end{split}
\end{equation*}
Using \cref{lem:PSDcmp} (2), since $m I \preceq A \preceq MI$, we obtain
\[
    m B \preceq B^{1/2} A B^{1/2} \preceq M B,
\]
As a result, using \cref{lem:PSDcmp} (4) for $\psi(x) = (1+x)^{-2}$,
\begin{equation*}
\begin{split}
    B^{1/2} ( I + MB )^{-2} B^{1/2} \preceq~& B^{1/2} ( I + B^{1/2} A B^{1/2})^{-2} B^{1/2} \\
    \preceq~& B^{1/2} ( I + mB )^{-2} B^{1/2}.
\end{split}
\end{equation*}
The conclusion follows by observing that the bounds are of functions of $B$ are thus the sequence of multiplication is unimportant. 

If $B$ is nonsingular, we can use the same argument to show  
\[
    B_\epsilon (I+M B_\epsilon)^{-2} \preceq (I+B_\epsilon A)^{-1} B_\epsilon (I+AB_\epsilon)^{-1} \preceq B_\epsilon (I+m B_\epsilon)^{-2}.
\]
with $B_\epsilon=B+\varepsilon I$ and take limit $\varepsilon\gto 0$.
\end{proof}

\section{Implementation of the linearization}
\label{app:linearize}
First notice it is not necessary to compute $H_n$ explicitly, since the DEKI scheme only requires its action on the ensemble deviation, i.e. $ H_n \tau_n^{(j)} = \begin{bmatrix} G_n \tau_n^{(j)} \\ \mcC_0^{-1/2} \tau_n^{(j)}
\end{bmatrix} $. So it suffices to compute $G_n \tau_n^{(j)}$. Denote 
\[
    T_n = [u_n^{(1)} -\mean{u}_n, \dots, u_n^{(J)} -\mean{u}_n] \in \mR^{d_u\times J},
\]
\[
    Y_n = [\mcG (u_n^{(1)}) -\mean{\mcG (u_n)}, \dots, \mcG (u_n^{(J)}) -\mean{\mcG (u_n)}] \in \mR^{d_y\times J},
\]
Perform the reduced QR decomposition on $T_n$ and $Y_n$: 
\[
    T_n = V_n Q_n, \quad V_n\in\mR^{d_u\times r}; \quad Y_n = W_n R_n, \quad W_n\in\mR^{d_y\times r'}.
\]
where $r$ and $r'$ are the ranks of $T_n$ and $Y_n$ respectively. Denote $\Lambda_n = Q_nQ_n\matT $, and then
\[
    (J-1)C_n^{uu} :=  T_n T_n\matT  = V_n Q_n Q_n\matT  V_n\matT  = V_n \Lambda_n V_n\matT, 
\]
\[
    (J-1)C_n^{yu} := Y_n T_n\matT  = W_n R_n Q_n\matT  V_n\matT, 
\]
\[
    \St ~ C_n^{yu} (C_n^{uu})^\dagger = W_n R_n Q_n\matT  \Lambda_n^{-1}  V_n\matT .
\]
It is direct to verify that the optimal $G_n$ to \cref{eq:linearize} is 
\[
    G_n = W_n \mcT_M( R_n Q_n\matT  \Lambda_n^{-1}) V_n\matT, 
\]
where $\mcT_M$ is the truncation operator on singular values: 
\[
    \mcT_M A = U \min\{ \Lambda , M \} V\matT , \quad \text{if } A = U \Lambda V\matT \text{ is the SVD}.
\]
Here $\min$ is performed entrieswise. Therefore, 
\[
    [G_n \tau_n^{(1)}, \dots, G_n \tau_n^{(J)}] = G_n T_n = W_n \mcT_M( R_n Q_n\matT  \Lambda_n^{-1}) Q_n. 
\]
The algorithm is summarized as below.
\begin{enumerate}
    \item Compute the deviation matrix $T_n$ and $Y_n$. 
    \item Perform the reduced QR decomposition: $T_n = V_n Q_n, Y_n = W_n R_n$.
    \item Compute $G_n T_n = W_n \mcT_M \Rectbrac{ R_n Q_n\matT  (Q_n Q_n\matT )^{-1} } Q_n$.
\end{enumerate}
The main computational cost is the reduced QR decomposition of $T_n,Y_n$ and  the SVD in the truncation operator. First notice that $R_n Q_n\matT \Lambda_n^{-1} \in \mR^{r' \times r}$, so that the SVD costs $\mcO(r^2 r') = \mcO(J^3)$ operations, which has no dimension dependence. For the reduced QR decomposition, since $T_n \in \mR^{d_u\times J}, Y_n \in \mR^{d_y\times J}$, the computational cost would be $\mcO( (d_u + d_y )J^2)$, which is linear in the dimension. Therefore, it is not a major computational cost in DEKI (see remarks after \cref{thm:convergence}).

\section{LEKI and localization}
\label{app:LEKI}
LEKI method is proposed in \cite{MR4580673}. We use the following discrete-time scheme:
\[
u_{n+1}^{(j)} = u_{n}^{(j)} + h_n \tilde{C}_n^{uz} \Brac{ \mcH(u_n^{(j)}) - z } + \sigma^2 \xi_n^{(j)}.
\]
The localized covariance is defined via a localization matrix $\Psi \in \mR^{d_z\times d_u}$ and
\[
\tilde{C}_n^{uz} = \Psi \circ C_n^{uz}.
\]
And $\xi_n^{(j)}$ is some artificial noise chosen following \cite{MR4580673} so that
\[
\frac{1}{J-1} \sum_{j=1}^J \Brac{ \xi_n^{(j)} \otimes \tau_n^{(j)} + \tau_n^{(j)} \otimes \xi_n^{(j)} } = \Sigma_n,
\]
where the diagonal entries of $\Sigma_n$ are all $1$. The noise level is fixed $\sigma=10^{-3}$.

For the linear transport model, we compare two different localization designs. Both uses the Gaspari Cohn function $\psi$ \cite{2001MWRv..129..123H}, and 
\[
\Psi(i,j) = \psi(|i-i_{\rm o}(j)|/r_{\rm \rm loc}),
\]
where $i_{\rm o}(j)$ is the coordinate of the parameter that mostly influences the $j$-th observation, and $r_{\rm loc}$ is the localization radius, see \cite{MR4580673}. The one with knowledge of the speed uses $r_{\rm loc} = 1.5$ and
\[
    i_{\rm o}(j) = \lfloor d_y (i/d_u + a  ) \rfloor.
\]
And the one without the knowledge simply chooses $i_{\rm o}(j) = j$ with the same localization radius.

\end{document}

%% file: siam_shared.tex
\headers{DEKI for high dimensional inverse problems}{Shuigen Liu, Sebastian Reich and Xin T. Tong}

\author{Shuigen Liu\thanks{National University of Singapore (\email{shuigen@u.nus.edu},\email{mattxin@nus.edu.sg})}
\and Sebastian Reich\thanks{Institut für Mathematik, Universität Potsdam (\email{sebastian.reich@uni-potsdam.de}).}
\and Xin T. Tong\footnotemark[2]}

%% file: DEKI.bbl
\begin{thebibliography}{10}

\bibitem{NIPS2009_2387337b}
{\sc A.~Agarwal, M.~J. Wainwright, P.~Bartlett, and P.~Ravikumar}, {\em
  Information-theoretic lower bounds on the oracle complexity of convex
  optimization}, in Advances in Neural Information Processing Systems,
  Y.~Bengio, D.~Schuurmans, J.~Lafferty, C.~Williams, and A.~Culotta, eds.,
  vol.~22, Curran Associates, Inc., 2009.

\bibitem{MR3826506}
{\sc M.~Benning and M.~Burger}, {\em Modern regularization methods for inverse
  problems}, Acta Numer., 27 (2018), pp.~1--111,
  \url{https://doi.org/10.1017/s0962492918000016}.

\bibitem{MR3988266}
{\sc D.~Bl\"{o}mker, C.~Schillings, P.~Wacker, and S.~Weissmann}, {\em Well
  posedness and convergence analysis of the ensemble {K}alman inversion},
  Inverse Problems, 35 (2019), pp.~085007, 32,
  \url{https://doi.org/10.1088/1361-6420/ab149c}.

\bibitem{calvello22}
{\sc E.~Calvello, S.~Reich, and A.~M. Stuart}, {\em Ensemble {Kalman} methods:
  {A} mean field perspective}, arXiv preprint arXiv:2209.11371,  (2022).

\bibitem{MR4163541}
{\sc Y.~Carmon, J.~C. Duchi, O.~Hinder, and A.~Sidford}, {\em Lower bounds for
  finding stationary points {I}}, Math. Program., 184 (2020), pp.~71--120,
  \url{https://doi.org/10.1007/s10107-019-01406-y}.

\bibitem{chada2018parameterizations}
{\sc N.~K. Chada, M.~A. Iglesias, L.~Roininen, and A.~M. Stuart}, {\em
  Parameterizations for ensemble {K}alman inversion}, Inverse Problems, 34
  (2018), p.~055009.

\bibitem{MR4089506}
{\sc N.~K. Chada, A.~M. Stuart, and X.~T. Tong}, {\em Tikhonov regularization
  within ensemble {K}alman inversion}, SIAM J. Numer. Anal., 58 (2020),
  pp.~1263--1294, \url{https://doi.org/10.1137/19M1242331}.

\bibitem{MR4405495}
{\sc N.~K. Chada and X.~T. Tong}, {\em Convergence acceleration of ensemble
  {K}alman inversion in nonlinear settings}, Math. Comp., 91 (2022),
  pp.~1247--1280, \url{https://doi.org/10.1090/mcom/3709}.

\bibitem{MR1408680}
{\sc H.~W. Engl, M.~Hanke, and A.~Neubauer}, {\em Regularization of inverse
  problems}, vol.~375 of Mathematics and its Applications, Kluwer Academic
  Publishers Group, Dordrecht, 1996.

\bibitem{pmlr-v48-gal16}
{\sc Y.~Gal and Z.~Ghahramani}, {\em Dropout as a {B}ayesian approximation:
  {R}epresenting model uncertainty in deep learning}, in Proceedings of The
  33rd International Conference on Machine Learning, vol.~48, New York, USA,
  2016, PMLR, pp.~1050--1059.

\bibitem{2001MWRv129.7690}
{\sc T.~M. Hamill, J.~S. Whitaker, and C.~Snyder}, {\em Distance-dependent
  filtering of background error covariance estimates in an ensemble {K}alman
  filter}, Monthly Weather Review, 129 (2001), pp.~2776 -- 2790,
  \url{https://doi.org/https://doi.org/10.1175/1520-0493(2001)129<2776:DDFOBE>2.0.CO;2}.

\bibitem{MR2978290}
{\sc R.~A. Horn and C.~R. Johnson}, {\em Matrix analysis}, Cambridge University
  Press, Cambridge, second~ed., 2013.

\bibitem{1998MWRv..126..796H}
{\sc P.~L. {Houtekamer} and H.~L. {Mitchell}}, {\em {Data Assimilation Using an
  Ensemble {K}alman Filter Technique}}, Monthly Weather Review, 126 (1998),
  p.~796,
  \url{https://doi.org/10.1175/1520-0493(1998)126<0796:DAUAEK>2.0.CO;2}.

\bibitem{2001MWRv..129..123H}
{\sc P.~L. {Houtekamer} and H.~L. {Mitchell}}, {\em {A Sequential Ensemble
  {K}alman Filter for Atmospheric Data Assimilation}}, Monthly Weather Review,
  129 (2001), p.~123,
  \url{https://doi.org/10.1175/1520-0493(2001)129<0123:ASEKFF>2.0.CO;2}.

\bibitem{MR3041539}
{\sc M.~A. Iglesias, K.~J.~H. Law, and A.~M. Stuart}, {\em Ensemble {K}alman
  methods for inverse problems}, Inverse Problems, 29 (2013), pp.~045001, 20,
  \url{https://doi.org/10.1088/0266-5611/29/4/045001}.

\bibitem{NIPS2015_bc731692}
{\sc D.~P. Kingma, T.~Salimans, and M.~Welling}, {\em Variational dropout and
  the local reparameterization trick}, in Advances in Neural Information
  Processing Systems, vol.~28, Curran Associates, Inc., 2015,
  \url{https://proceedings.neurips.cc/paper_files/paper/2015/file/bc7316929fe1545bf0b98d114ee3ecb8-Paper.pdf}.

\bibitem{MR4226141}
{\sc A.~Kirsch}, {\em An introduction to the mathematical theory of inverse
  problems}, vol.~120 of Applied Mathematical Sciences, Springer, Cham, [2021]
  \copyright 2021, \url{https://doi.org/10.1007/978-3-030-63343-1}.
\newblock Third edition [of 1479408].

\bibitem{MR3998631}
{\sc N.~B. Kovachki and A.~M. Stuart}, {\em Ensemble {K}alman inversion: {A}
  derivative-free technique for machine learning tasks}, Inverse Problems, 35
  (2019), pp.~095005, 35, \url{https://doi.org/10.1088/1361-6420/ab1c3a}.

\bibitem{MR4412181}
{\sc C.~Liu, L.~Zhu, and M.~Belkin}, {\em Loss landscapes and optimization in
  over-parameterized non-linear systems and neural networks}, Appl. Comput.
  Harmon. Anal., 59 (2022), pp.~85--116,
  \url{https://doi.org/10.1016/j.acha.2021.12.009}.

\bibitem{MR702836}
{\sc A.~S. Nemirovsky and D.~B.~a. Yudin}, {\em Problem complexity and method
  efficiency in optimization}, Wiley-Interscience Series in Discrete
  Mathematics, John Wiley \& Sons, Inc., New York, 1983.
\newblock Translated from the Russian and with a preface by E. R. Dawson.

\bibitem{MR2142598}
{\sc Y.~Nesterov}, {\em Introductory lectures on convex optimization}, vol.~87
  of Applied Optimization, Kluwer Academic Publishers, Boston, MA, 2004,
  \url{https://doi.org/10.1007/978-1-4419-8853-9}.
\newblock A basic course.

\bibitem{MR3627456}
{\sc Y.~Nesterov and V.~Spokoiny}, {\em Random gradient-free minimization of
  convex functions}, Found. Comput. Math., 17 (2017), pp.~527--566,
  \url{https://doi.org/10.1007/s10208-015-9296-2}.

\bibitem{doi:10.3402/tellusa.v56i5.14462}
{\sc E.~Ott, B.~R. Hunt, I.~Szunyogh, A.~V. Zimin, E.~J. Kostelich, M.~Corazza,
  E.~Kalnay, D.~Patil, and J.~A. Yorke}, {\em A local ensemble {K}alman filter
  for atmospheric data assimilation}, Tellus A: Dynamic Meteorology and
  Oceanography, 56 (2004), pp.~415--428,
  \url{https://doi.org/10.3402/tellusa.v56i5.14462}.

\bibitem{MR4487558}
{\sc F.~Parzer and O.~Scherzer}, {\em On convergence rates of adaptive ensemble
  {K}alman inversion for linear ill-posed problems}, Numer. Math., 152 (2022),
  pp.~371--409, \url{https://doi.org/10.1007/s00211-022-01314-y}.

\bibitem{MR4568201}
{\sc J.~Pidstrigach and S.~Reich}, {\em Affine-invariant ensemble transform
  methods for logistic regression}, Found. Comput. Math., 23 (2023),
  pp.~675--708, \url{https://doi.org/10.1007/s10208-022-09550-2}.

\bibitem{MR3654885}
{\sc C.~Schillings and A.~M. Stuart}, {\em Analysis of the ensemble {K}alman
  filter for inverse problems}, SIAM J. Numer. Anal., 55 (2017),
  pp.~1264--1290, \url{https://doi.org/10.1137/16M105959X}.

\bibitem{MR3764752}
{\sc C.~Schillings and A.~M. Stuart}, {\em Convergence analysis of ensemble
  {K}alman inversion: {T}he linear, noisy case}, Appl. Anal., 97 (2018),
  pp.~107--123, \url{https://doi.org/10.1080/00036811.2017.1386784}.

\bibitem{JMLR:v15:srivastava14a}
{\sc N.~Srivastava, G.~Hinton, A.~Krizhevsky, I.~Sutskever, and
  R.~Salakhutdinov}, {\em Dropout: {A} simple way to prevent neural networks
  from overfitting}, Journal of Machine Learning Research, 15 (2014),
  pp.~1929--1958, \url{http://jmlr.org/papers/v15/srivastava14a.html}.

\bibitem{MR2652785}
{\sc A.~M. Stuart}, {\em Inverse problems: {A} {B}ayesian perspective}, Acta
  Numer., 19 (2010), pp.~451--559,
  \url{https://doi.org/10.1017/S0962492910000061}.

\bibitem{MR4580673}
{\sc X.~T. Tong and M.~Morzfeld}, {\em Localized ensemble {K}alman inversion},
  Inverse Problems, 39 (2023), pp.~Paper No. 064002, 38.

\bibitem{10.5555/2999611.2999651}
{\sc S.~Wager, S.~Wang, and P.~Liang}, {\em Dropout training as adaptive
  regularization}, in Proceedings of the 26th International Conference on
  Neural Information Processing Systems - Volume 1, NIPS'13, Red Hook, NY, USA,
  2013, Curran Associates Inc., p.~351–359.

\bibitem{pmlr-v28-wan13}
{\sc L.~Wan, M.~Zeiler, S.~Zhang, Y.~Le~Cun, and R.~Fergus}, {\em
  Regularization of neural networks using {D}rop{C}onnect}, in Proceedings of
  the 30th International Conference on Machine Learning, vol.~28, Atlanta,
  Georgia, USA, 2013, PMLR, pp.~1058--1066.

\bibitem{MR4482475}
{\sc S.~Weissmann}, {\em Gradient flow structure and convergence analysis of
  the ensemble {K}alman inversion for nonlinear forward models}, Inverse
  Problems, 38 (2022), pp.~Paper No. 105011, 30.

\bibitem{weissmann2024almost}
{\sc S.~Weissmann, S.~Klein, W.~Azizian, and L.~D{\"o}ring}, {\em Almost sure
  convergence rates of stochastic gradient methods under gradient domination},
  arXiv preprint arXiv:2405.13592,  (2024).

\end{thebibliography}
